\documentclass[11pt,a4paper]{amsart}
\usepackage{amssymb,amsfonts,amsmath,mathrsfs}
\usepackage[left=2cm,right=2cm,top=2cm,height=250mm]{geometry}
\usepackage{color}
\usepackage[colorlinks=true,linkcolor=blue,citecolor=red, pdfborder={0 0 0}]{hyperref}

\numberwithin{equation}{section}
\newtheorem{theorem}{Theorem}[section]
\newtheorem{lemma}{Lemma}[section]
\newtheorem{Def}{Definition}[section]
\newtheorem{rem}{Remark}[section]
\newtheorem{prop}{Proposition}[section]
\newtheorem{cor}{Corollary}[section]

\def\R{\Bbb R}
\def\Dx{\Delta_x}

\def\Dt{\partial_t}
\def\({\left(}
\def\){\right)}
\def\eb{\varepsilon}
\def\Cal{\mathcal}

\begin{document}
\title[Sub-quintic wave equation with fractional damping]{Strichartz estimates and smooth attractors for a sub-quintic wave equation with fractional damping in bounded domains}
\author[] {Anton Savostianov}

\begin{abstract}
The work is devoted to Dirichlet problem for sub-quintic semi-linear wave equation with damping damping term of the form $(-\Dx)^\alpha\Dt u$, $\alpha\in(0,\frac{1}{2})$, in bounded smooth domains of $\Bbb R^3$. It appears that to prove well-posedness and develop smooth attractor theory for the problem we need additional regularity of the solutions, which does not follow from the energy estimate. Considering the original problem as perturbation of the linear one the task is reduced to derivation of Strichartz type estimate for the linear wave equation with fractional damping, which is the main feature of the work. Existence of smooth exponential attractor for the natural dynamical system associated with the problem is also established.
\end{abstract}


\address{University of Surrey, Department of Mathematics, Guildford, GU2 7XH, United Kingdom.}
\email{a.savostianov@surrey.ac.uk}

\maketitle

\section{Introduction}
In this work we consider the following semi-linear damped wave equation in a bounded smooth domain $\Omega\subset \R^3$
\begin{equation}
\label{eq intro}
\begin{cases}
\Dt^2 u-\Dx u+\gamma(-\Dx)^\alpha\Dt u+f(u)=g,\ x\in\Omega,\\
u|_{\partial\Omega}=0,\ u|_{t=0}=u_0,\ \Dt u|_{t=0}=u_1, 
\end{cases}
\end{equation}
where constants $\gamma>0$ and $\alpha\in(0,\frac{1}{2})$, initial data $(u_0,u_1)$ belong to standard energy space $\Cal{E}=H^1_0(\Omega)\times L^2(\Omega)$, external force $g=g(x)\in L^2(\Omega)$,  and non-linearity $f\in C^1(\R)$ is of sub-quintic
growth \eqref{f ass1} and satisfies natural dissipative assumptions \eqref{f ass2}.
Also for brevity we use the notation $\xi_u(t):=(u(t),\Dt u(t))$ and $\|\cdot\|$ for norm in $L^2(\Omega)$.

In last years wave equation with fractional damping term attracts more and more attention (see for example \cite{chen,tri1,tri2,Chu2010,KZwvEq2009,AZqdw}). From applied point of view this is related to the fact that such equations model various processes with frequency depending attenuation ( see \cite{chen,tree} and references therein). From mathematical point of view even in linear case ( $f\equiv 0$) properties of such equations demonstrate non-trivial dependence on $\alpha$. For example, linear equation \eqref{eq intro} generates an analytic semigroup iff $\alpha\in[\frac{1}{2},1]$ (see \cite{tri2}). For $\alpha\in(0,1)$ equation \eqref{eq intro} possesses smoothing property similar to parabolic equations. For $\alpha=1$ smoothing property is instantaneous for $\Dt u$ and asymptotic for $u$ ( see \cite{CV}). It is well known that when $\alpha=0$ equation \eqref{eq intro} enjoys finite speed of propagation property. 

%which are summarized in the table 
%\begin{center}
%	\begin{tabular}{|c|c|c|c|}
%	\hline
%	$\theta$ & Semigroup & Smoothing  & Maximal Regularity \\
%	\hline
%	0 & $C_0$ & Asymptotic & No \\
%	\hline
%	$\ (0,\frac{1}{2}) \ $ & $C^\infty$ & Instantaneous & No\\
%	\hline
%	$\ [\frac{1}{2},1) \ $ & Analytic & Instantaneous & Yes \\
%	\hline
%	$1$ & Analytic & \mbox{ }Instantaneous for $\Dt u$,\mbox{ } & Yes\\
%	& & asymptotic for $u$ & \\ \hline
%	\end{tabular}
%\end{center}
%These results can be found in \cite{carc3,PataZel2006,tri1,tri2}.

In the presence of the non-linear term of type $f(u)\sim u|u|^q$ even well-posedness becomes questionable. For a long time the qubic growth ( $q=2$) of the non-linearity was considered as critical for weakly damped ( $\alpha=0$) wave equation ( see \cite{GraPat,PataZelwd} ) and quintic growth rate ( $q=4$) was considered as critical for strongly damped ( $\alpha=1$) wave equation (see \cite{PataZel2006}). However, a breakthrough was done in \cite{KZwvEq2009}, where global wellposedness, dissipativity and existence of smooth attractor was obtained for $q\in[0,\infty)$ as long as $\alpha\in [\frac{3}{4},1]$. In addition, for $\alpha\in[\frac{1}{2},\frac{3}{4})$ there was obtained well-posedness and built up a smooth attractor theory at least under assumption $0\leq q<q(\alpha)=\frac{8\alpha}{3-4\alpha}$. For $\alpha=\frac{1}{2}$ the critical case with $q=q(\frac{1}{2})=4$ ( that is nonlinearity has quintic growth) is done in
\cite{AZqdw}. It appears that in this case the solution possesses hidden extra regularity, $u\in L^2([0,T];H^\frac{3}{2}(\Omega))$, that does not follow directly from energy estimate, but from hidden Lyapunov type functional.

The progress with weakly damped equation is closely related to the progress with pure wave equation. In the first part of 90's it was noticed that for linear wave equation $L^p([0,T];L^q(\Omega))$ norm ( for some admissible $p$ and $q$) can be controlled via energy norm of initial data and external force when $\Omega=\R^3$.
Such type of estimates became known as Strichartz estimates. Furthermore, this type of estimates allow to establish existence of more regular energy solutions ( with finite $L^p([0,T];L^q(\Omega))$ norm) in semilinear case (see \cite{SS,kap,kap1}).  In contrast to the case $\alpha=\frac12$ it is not known whether all energy solutions satisfy Strichartz estimates and so we will refer to this class of solutions as Shatah-Struwe solutions. Partial results highlighting this question are available in \cite{plan3}. The above mentioned additional regularity allowed to prove uniqueness of Shatah-Struwe solutions for semi-linear wave equation in the whole space $\R^3$ in the case when $q\leq 4$. However Strichartz estimates in {\it bounded smooth} domains became available only recently. This, in turn, leads to well-posedness of quintic wave equation ( see \cite{Sogge2009,stri}) and as consequence of quintic weakly damped wave equation ( see \cite{KSZ}) in smooth bounded domains. %We also note that in contrast to the case $\alpha=\frac{1}{2}$ where the additional regularity was proven for all energy solutions, here we know only existence of energy solutions with additional regularity and  the question whether all energy solutions possess this regularity is mostly open (see \cite{plan3}).

Smooth attractor theory for quintic weakly damped wave equation is developed in \cite{KSZ}. The main difficulty arising in this case is the fact, that despite finiteness of $L^4([t,t+1];L^{12}(\Omega))$ norm of the solution we do not have any explicit control of this norm as $t\to\infty$ since its finiteness is gained by contradiction arguments and thus this regularity, apriori, could be lost at infinity. Fortunately, as shown in \cite{KSZ} this is not the case. The main tool used to obtain this result is the theory of weak attractors, and in particular, the backward smoothing property on weak attractor which are obtained in \cite{ZelDCDS}. A brief review highlighting the main ideas of smooth attractor theory in critical case $q=4$ for $\alpha=\frac{1}{2}$ and $\alpha=0$ in bounded domains are presented in \cite{AZr}. 

Even less is known for the wave equation \eqref{eq intro} with damping term $(-\Dx)^\alpha\Dt u$ when $\alpha\in (0,\frac{1}{2})$. To the best of our knowledge, well-posedness and attractor theory in the case of cubic non-linearity can be done similarly to the case of $\alpha=0$ ( see \cite{CV}). However, there are no such results in case of super-cubic non-linearity. 

The aim of this work is twofold. The first one is to prove existence and uniqueness of Shatah-Struwe solutions for semi-linear wave equation \eqref{eq intro} with damping term $(-\Dx)^\alpha\Dt u$, when $\alpha\in(0,\frac{1}{2})$ and subquintic non-linearity in a bounded smooth domain. The second one is to build up smooth attractor theory for the considered equation. The main problem in the case $\alpha\in(0,\frac12)$ is derivation of control of $L^5([0,T];L^{10}(\Omega))$ norm to solutions of the linear problem
\begin{equation}
\label{eq ldwv}
\begin{cases}
\Dt^2u+\gamma(-\Dx)^\alpha\Dt u-\Dx u=h(t),\\ 
u|_{t=0}=u_0,\ \Dt u|_{t=0}=u_1,\ u|_{\partial\Omega}=0,
\end{cases} 
\end{equation}
where $(u_0,u_1)\in H^1_0(\Omega)\times L^2(\Omega)$ and $h\in L^1([0,T];L^2(\Omega))$. In contrast to the case $\alpha=0$, we can not just put damping term $(-\Dx)^\alpha\Dt u$ to the right hand side and use usual Strichartz estimate since
$(-\Dx)^\alpha\Dt u\in L^2([0,T];H^{-\alpha})$ but not in $L^1([0,T];L^2(\Omega))$. To overcome this difficulty we notice
that change of variables $v(t)=e^{\frac{\gamma}{2}(-\Dx)^\alpha t}u(t)$
transforms linear homogeneous damped wave equation \eqref{eq ldwv}( with $h\equiv 0$) into the following one
\begin{equation}
\label{eq lwvm}
\Dt^2 v-\Dx v-\frac{\gamma^2}{4}(-\Dx)^{2\alpha}v=0.
\end{equation}
Then using spectral cluster estimates obtained in \cite{Lp clustrs} and adapting technique presented in \cite{stri} we are still able to tackle extra term $\frac{\gamma^2}{4}(-\Dx)^{2\alpha}v$ and establish control of $L^5([0,T];L^{10}(\Omega))$ norm for solutions of \eqref{eq lwvm}. This implies the control of $L^5([0,T];L^{10}(\Omega))$ norm for homogeneous equation \eqref{eq ldwv} since operators $e^{-\frac{\gamma}{2}(-\Dx)^\alpha t} $ are bounded from $L^p(\Omega)$ to $L^p(\Omega)$ for $p\in (1,\infty)$ as long as $t\geq 0$. Consequently this also gives $L^5([0,T];L^{10}(\Omega))$ control for non-homogeneous equation \eqref{eq ldwv} that leads to the first main result
\begin{theorem}( see Proposition \ref{prop StrLin})
Let $\gamma$ be a strictly positive number, $h(t)\in L^1([0,T];L^2(\Omega))$ and initial data $(u_0,u_1)\in H^1_0(\Omega)\times L^2(\Omega)$. Then every energy solution $u$ to problem \eqref{eq ldwv} possesses the following extra regularity
\begin{equation}
\|u\|_{L^5([0,T];L^{10}(\Omega))}\leq C(\|u_0\|_{H^1(\Omega)}+\|u_1\|+\|h\|_{L^1([0,T];L^2(\Omega))}),
\end{equation}
where constant $C$ is independent of $T$ and initial data $(u_0,u_1)$.
\end{theorem}
Finally, considering equation \eqref{eq intro} as perturbation of \eqref{eq ldwv} we prove global existence and uniqueness of Shatah-Struwe solutions for original problem \eqref{eq intro} with finite $L^5([0,T];L^{10}(\Omega))$ norm
\begin{theorem} (see Theorems \ref{th StrEx} and \ref{th uniq.cont})
Let $\gamma>0$, and nonlinearity $f\in C^1(\R)$ is such that
\begin{equation}
 f(u)u\geq -C,\quad |f'(u)|\leq C(1+|u|^q),\quad q\in[0,4).
\end{equation} 
Then for any initial data $\xi_0=(u_0,u_1)\in H^1_0(\Omega)\times L^2(\Omega)$ there exists and unique Shatah-Struwe solution $u$ of problem \eqref{eq intro} with finite $L^5([0,T];L^{10}(\Omega))$ norm. Moreover this solution enjoys the following estimate
\begin{multline}
\|\xi_u(t)\|_\Cal{E}+\|\Dt u\|_{L^2([\max\{0,t-1\},t];H^\alpha(\Omega))}+\\
\|u\|_{L^5([\max\{0,t-1\},t];L^{10}(\Omega))}\leq Q(\|\xi_0\|_\Cal{E})e^{-\beta t}+Q(\|g\|),\quad t\geq 0,
\end{multline}
where constant $\beta>0$ and increasing function $Q$ are independent of $t$ and $\xi_0$.  
\end{theorem}
We note that available spectral cluster estimates obtained in \cite{Lp clustrs} allow us to get the control of $L^5([0,T];L^{10}(\Omega))$ norm at most that allows easily to tackle sub-quintic growth rate of non-linearity but requires further work for {\it quintic} non-linearity. Indeed, in case of quintic non-linearity, we can prove only local existence of Shatah-Struwe solutions and non-concentration of $L^5([0,T];L^{10}(\Omega))$ remains open. The main difficulty here is related to the fact that in case $\alpha\in(0,\frac{1}{2})$ the finite speed of propagation fails. And thus arguments from \cite{stri} do not work directly.

The attractor theory for the considered equation and $q\in[0,4)$ is built up due to the fact that the considered Shatah-Struwe solutions possess smoothing property similar to parabolic equations ( see Theorem \ref{th sm3}).

The work is organised as follows. In Section 2 we derive a Strichartz type estimate for the linear problem \eqref{eq ldwv}. Existence and uniqueness of Shatah-Struwe solutions in semilinear case, as well as their basic properties, are established in Section 3. In Section 4 we show that Shatah-Struwe solutions to problem \eqref{eq intro} satisfy parabolic-like smoothing property. Finally, existence of a smooth global attractor as well as exponential global attractor is given in Section 5. 

\section{Strichartz estimate in the linear case}
In this section we establish Strichartz estimates for the linear damped wave equation \eqref{eq ldwv},
%\begin{equation}
%\label{eq ldwv}
%\begin{cases}
%\Dt^2u+\gamma(-\Dx)^\alpha\Dt u-\Delta u =h(t), \quad x\in \Omega\\
%u|_{\partial{\Omega}}=0, \quad u(0)=u_0\in H^1_0(\Omega) ,\quad \Dt u(0)=u_1\in L^2(\Omega)
%\end{cases}
%\end{equation}
where $\Omega\subset \R^3$ is a bounded smooth domain, $\gamma>0$ is a constant and $h(t)\in L^1\([0,T];L^2(\Omega)\)$ and $\alpha\in(0,\frac{1}{2})$.

Let us remind the classical energy estimate for linear equation \eqref{eq ldwv}
\begin{prop}
\label{prop lin basic}
Let $u$ be a distributional solution of \eqref{eq ldwv} and the above assumptions hold. Then $u$ satisfies the following estimates
\begin{equation}
\label{2.0}
\|\xi_u(t)\|_\Cal{E}+\|\Dt u(s)\|_{L^2([\max\{0,t-1\},t];H^\alpha(\Omega))}ds\leq C\(e^{-\beta t}\|\xi_0\|_\Cal{E}+\int_0^te^{-\beta(t-s)}\|h(s)\|ds\),
\end{equation}
%\begin{equation}
%\label{2.01}
%\leq C\(\|\xi_0\|^2_\Cal{E}+\|h\|^2_{L^1([0,T];L^2(\Omega))}\), 
%\end{equation}
for some small enough $\beta>0$ and $C>0$ which depend on $\gamma$ only.
\end{prop}
Estimate \eqref{2.0} easily follows from multiplication of \eqref{eq ldwv} by $\Dt u+ru$ with small enough constant $r>0$. 

\begin{cor}
\label{cor lin L2(H 1+a)}
Let assumptions of Proposition \ref{prop lin basic} be satisfied and $u$ be a distributional solution of \eqref{eq ldwv}. Then, in addition, $u\in L^2([0,T];H^{1+\alpha}(\Omega))$ for any $T>0$ and the following estimate holds
\begin{equation}
\label{est lin L2(H 1+a)}
\|u(s)\|_{L^2([\max\{0,t-1\},t];H^{1+\alpha}(\Omega))}\leq C\(e^{-\beta t}\|\xi_0\|_\Cal{E}+\int_0^t\|h(s)\|e^{-\beta(t-s)}ds\),\ t\geq 0,
\end{equation}
for some $C,\beta>0$, which is independent of $t$, $\xi_0$ and $h(t)$.
\end{cor}
\begin{proof}
The corollary easily follows from multiplication of \eqref{eq ldwv} by $(-\Dx)^\alpha u$ and Proposition \ref{prop lin basic}. Indeed, from \eqref{2.0} and $h\in L^1([0,T];L^2(\Omega))$ follows that product $(h(t),(-\Dx)^\alpha u)$ makes sense. To perform multiplication involving linear terms one can apply projector $P_N$ on the first $N$ eigenfunctions of $-\Dx$ to \eqref{eq ldwv}, multiply the obtained equation by $(-\Dx)^\alpha u_N$, where $u_N=P_N u$ and pass to the limit in subsequent estimates. Below we derive \eqref{est lin L2(H 1+a)} in formal way.

Multiplication of \eqref{eq ldwv} by $(-\Dx)^\alpha u$ gives
\begin{equation}
\frac{d}{dt}\((\Dt u,(-\Dx)^\alpha u)+\frac{\gamma}{2}\|(-\Dx)^\alpha u\|^2\)+\|(-\Dx)^\frac{1+\alpha}{2} u\|^2=(h,(-\Dx)^\alpha u)+\|(-\Dx)^\frac{\alpha}{2}\Dt u\|^2.
\end{equation}
Integrating the above inequality from $\tau(t)=\max\{0,t-1\}$ to $t$ and using Cauchy inequality one finds
\begin{multline}
\int_{\tau(t)}^t\|(-\Dx)^\frac{1+\alpha}{2}u(s)\|^2ds\leq C\(\int_{\tau(t)}^t\|h(s)\|\|u(s)\|_{H^1(\Omega)}ds+\int_{\tau(t)}^t\|\Dt u(s)\|^2_{H^\alpha(\Omega)}ds+\right.\\
\Bigg.\|\xi_u(\tau(t))\|^2_{\Cal E}+\|\xi_u(t)\|^2_{\Cal E}\Bigg).
\end{multline}
Taking into account \eqref{2.0}, the we derive
\begin{multline}
\int_{\tau(t)}^t\|(-\Dx)^\frac{1+\alpha}{2}u(s)\|^2ds\leq C\sup_{s\in[\tau(t),t]}\|u(s)\|_{H^1(\Omega)}\int_{\tau(t)}^t\|h(s)\|ds+\\C\(e^{-\beta t}\|\xi_0\|_\Cal{E}+\int_0^te^{-\beta(t-s)}\|h(s)\|ds\)^2\leq C\(e^{-\beta t}\|\xi_0\|_\Cal{E}+\int_0^te^{-\beta(t-s)}\|h(s)\|ds\)^2,
\end{multline}
that completes the proof.
\end{proof}
However to consider semi-linear damped wave equation with sub-quintic non-linearity we will need additional space-time regularity. Following the arguments from \cite{stri} we get this regularity from $L^p$ esteimates on spectral clusters obtained in
\cite{Lp clustrs} (Theorem 7.1), namely
\begin{theorem}\label{th Sogge}
Let $\Omega\subset\R^3$ be a smooth bounded domain, $\{e_k\}_{k=1}^\infty$ and $\{\lambda_k\}_{k=1}^\infty$ are eigenfunctions and eigenvalues of $-\Delta$ respectively and $\mathbf{P}_\lambda=\mathbf{1}_{\sqrt{-\Delta}\in[\lambda,\lambda+1]}$, that is $\mathbf{P}_\lambda$ is the spectral projector on those eigenfunctions $e_k$ that $\sqrt{\lambda_k}\in[\lambda,\lambda+1)$.
Then
\begin{equation}
\label{sp est}
\left\|\mathbf{P}_\lambda u\right\|_{L^5(\Omega)}\leq C\lambda^\frac{2}{5}\|u\|.
\end{equation} 
for some absolute constant $C$ that depends on $\Omega$ only.
\begin{rem}
Estimate \eqref{sp est} is highly non-trivial and its proof is strongly based on harmonic analysis tools. To understand why it is remarkable let us compare this result with Sobolev's embedding. Due to continuous embedding $H^\frac{9}{10}(\Omega)\subset L^5(\Omega)$ it is easy to see that $\|\mathbf{P}_\lambda u\|_{L^5(\Omega)}\leq C\|(-\Dx)^\frac{9}{20}\mathbf{P}_\lambda u\|\leq C\lambda^\frac{9}{10}\|u\|$. Therefore we see that estimate \eqref{sp est} gains us additional $\frac{1}{2}$ in exponent growth that is crucial point in obtaining Strichartz estimates.
\end{rem}
\end{theorem} 
As usual first one need to obtain the desired estimates for homogeneous equation \eqref{eq ldwv}.
To this end we prove some auxiliary but crucial result (along the lines of \cite{stri})
\begin{prop}\label{prop eitDx est}
Assume $\alpha\in(0,\frac{1}{2})$ and for some $2\leq q< \infty $ the spectral projector $\mathbf{P}_\lambda$ satisfies
\begin{equation}
\label{sp asmtn}
\|\mathbf{P}_\lambda u\|_{L^q(\Omega)}\leq\lambda^\delta\|u\|,\ \mbox{ where }\delta>0. 
\end{equation}
Then for any $u_0\in H^{\delta+\frac{1}{2}-\frac{1}{q}}$ the function $v(t,x)=e^{it\sqrt{-\Dx-\frac{\gamma^2}{4}(-\Dx)^{2\alpha}}}u_0$ belongs to $L^q([0,2\pi];\Omega)$ and the following estimate holds
\begin{equation}
\|v\|_{L^q([0,2\pi];\Omega)}\leq C\|u_0\|_{H^{\delta+\frac{1}{2}-\frac{1}{q}}(\Omega)},
\end{equation}
where $C$ depends on $\gamma$ and $\alpha$ only.
\end{prop}
%\begin{rem}\label{rem sp ineq}
%Let $\mathbf{P}_{[a,b)}:=\mathbf{1}_{\sqrt{-\Dx}\in[a,b)}$, i.e. $\mathbf{P}_{[a,b)}$ is the projector on those eigenfunctions $e_k$ of $-\Dx$ that $\sqrt{\lambda_k}\in[a,b)$. Then \eqref{sp asmtn} implies the estimate 
%\begin{equation}
%\|\mathbf{P}_{[a,b)}u\|_{L^q(\Omega)}\leq (b-a+1)b^\delta\|u\|.
%\end{equation}
%\end{rem}
%Indeed,
%\begin{multline}
%\|\mathbf{P}_{[a,b)}u\|_{L^q(\Omega)}=\|\sum_{i=0}^{[b-a]-1}\mathbf{P}_{a+i}u+\mathbf{\mathbf{P}_{[a+[b-a],b)}}\|_{L^q(\Omega)}\leq\sum_{i=0}^{[b-a]-1}\|\mathbf{P}_{a+i}u\|_{L^q(\Omega)}+\|\mathbf{P}_{a+[b-a]}\mathbf{P}_{b-1}u\|_{L^q(\Omega)}\leq\\
%\sum_{i=0}^{[b-a]-1}(a+i)^\delta\|u\|+(a+[b-a])^\delta\|u\|\leq ([b-a]+1)(a+[b-a])^\delta\|u\|\leq(b-a+1)b^\delta\|u\|.
%\end{multline}
\begin{rem}
Obviously, since $\mathbf{P}_\lambda$ is a projector, \eqref{sp asmtn} implies
\begin{equation}
\|\mathbf{P}_\lambda u\|_{L^q(\Omega)}\leq\lambda^\delta\|\mathbf{P}_\lambda u\|, 
\end{equation}
that will be used below.
\end{rem}
\begin{proof}[Proof of Proposition \ref{prop eitDx est}.]
Let us define an abstract self-adjoint operator by 
\begin{equation}
Ae_n=\left[\sqrt{k^2-\frac{\gamma^2}{4}k^{4\alpha}}\right]e_n,\quad\mbox{ for those }n\geq 1:\sqrt{\lambda_n}\in[k,k+1),
\end{equation}
where $[\ \cdot\ ]$ denotes the integer part of a number and $\{e_n\}_{n=1}^\infty $ are eigenfunctions of $-\Dx $ with Dirichlet boundary conditions. First we want to prove the estimate for $\tilde{v}(t)=e^{itA}u_0 $. To this end it is convenient to estimate lower and higher Fourier modes separately, that is we represent $\tilde{v}(t)$ in the form 
\begin{equation}
\tilde{v}(t)=\Cal{P}_N\tilde{v}(t)+\Cal{Q}_N\tilde{v}(t)=e^{itA}\Cal{P}_Nu_0+e^{itA}\Cal{Q}_Nu_0,
\end{equation}
where $\Cal{P}_N$ is orthoprojector on the first $N$ eigenfunctions of $-\Dx$ and $\Cal{Q}_N=Id-\Cal{P}_N$. We set $N$ to be such a number that $\left[\sqrt{[\sqrt{\lambda_n}]^2-\frac{\gamma^2}{4}[\sqrt{\lambda_n}]^{4\alpha}}\right]\geq K$ for all $n\geq N+1$ and $K$ is large enough to be fixed bellow.

The estimate of lower modes is simple since we have only finite number of them. Indeed,
\begin{multline}
\|\Cal{P}_N\tilde{v}(t)\|_{L^q([0,2\pi];\Omega)}\leq\|\|\sum_{n=1}^Ne^{-t\left[\sqrt{\frac{\gamma^2}{4}[\sqrt{\lambda_n}]^{4\alpha}-[\sqrt{\lambda_n}]^2}\right]}(u_0,e_n)e_n\|_{L^q(\Omega)}\|_{L^q([0,2\pi])}\leq\\
\|\sum_{n=1}^N|(u_0,e_n)|\|e_n\|_{L^q(\Omega)}\|_{L^q([0,2\pi])}\leq \|C_{N,q}\sum_{n=1}^N|(u_0,e_n)|\|_{L^q([0,2\pi])}\leq (2\pi)^\frac{1}{q}C_{N,q}\sum_{n=1}^N|(u_0,e_n)|\leq\\
(2\pi)^\frac{1}{q}C_{N,q}\sqrt{N}\sqrt{\sum_{n=1}^N|(u_0,e_n)|^2}\leq (2\pi)^\frac{1}{q}C_{N,q}\sqrt{N} \|\Cal{P}_Nu_0\|,   
\end{multline}
where $C_{N,q}=max_{n=\overline{1,n}}\|e_n\|_{L^q(\Omega)}$. That is
\begin{equation}
\label{2.02}
\|\Cal{P}_N\tilde{v}(t)\|_{L^q([0,2\pi];\Omega)}\leq C_{N,q}\|\Cal{P}_Nu_0\|, 
\end{equation}
for some $C_{N,q}>0$.

The estimate of higher modes is more delicate. Writing down $\Cal{Q}_N\tilde{v}$ via eigenfunctions of $-\Dx$ we see that
\begin{equation}
\Cal{Q}_N\tilde{v}(t,x)=\sum_{m=N+1}^\infty e^{it\left[\sqrt{m^2-\frac{\gamma^2}{4}m^{4\alpha}}\right]}\mathbf{P}_mu_0=\sum_{k=K}^\infty e^{itk}\sum_{k\leq\sqrt{m^2-\frac{\gamma^2}{4}m^{4\alpha}}<k+1}\mathbf{P}_mu_0=: \sum_{k=K}^\infty e^{itk}\tilde{v}_k(x).
\end{equation}
Hence according to Plancherel's formula for a fixed x we get
\begin{equation}
\|\Cal{Q}_N\tilde{v}(\cdot,x)\|^2_{H^s(0,2\pi)}=2\pi\sum_{k=K}^\infty(1+k^2)^s|\tilde{v}_k(x)|^2.
\end{equation}
Now using that in 1 dimensional case $H^{s_0}(0,2\pi)\subset L^q(0,2\pi)$ for $s_0=\frac{1}{2}-\frac{1}{q}$ we deduce
\begin{multline}
\|\Cal{Q}_N\tilde{v}\|^2_{L^q([0,2\pi];\Omega)}=\| \|\Cal{Q}_N\tilde{v}\|_{L^q(0,2\pi)}\|^2_{L^q(\Omega)}\leq C\|\|\Cal{Q}_N\tilde{v}\|_{H^{s_0}(0,2\pi)}\|^2_{L^q(\Omega)}=\\
C\| \|\Cal{Q}_N\tilde{v}\|^2_{H^{s_0}(0,2\pi)}\|_{L^{\frac{q}{2}}(\Omega)}\leq C\|\sum_{k=K}^\infty(1+k^2)^{\frac{1}{2}-\frac{1}{q}}|\tilde{v}_k(x)|^2\|_{L^\frac{q}{2}(\Omega)}.
\end{multline}
Since $q\geq 2$, by Minkowski inequality we obtain
\begin{equation}
\label{0.9}
\|\Cal{Q}_N\tilde{v}\|^2_{L^q([0,2\pi];\Omega)}\leq C\sum_{k=K}^\infty(1+k^2)^{\frac{1}{2}-\frac{1}{q}}\||\tilde{v}_k(x)|^2\|_{L^\frac{q}{2}(\Omega)}=C\sum_{k=K}^\infty(1+k^2)^{\frac{1}{2}-\frac{1}{q}}\|\tilde{v}_k\|^2_{L^q(\Omega)}.
\end{equation}
To estimate $\|\tilde{v}_k(x)\|^2_{L^q(\Omega)}$ we notice that
\begin{multline}
\sqrt{(m+1)^2-\frac{\gamma^2}{4}(m+1)^{4\alpha}}-\sqrt{m^2-\frac{\gamma^2}{4}m^{4\alpha}}=\\
\frac{2m+1-\frac{\gamma^2}{4}\((m+1)^{4\alpha}-m^{4\alpha}\)}{\sqrt{(m+1)^2-\frac{\gamma^2}{4}(m+1)^{4\alpha}}+\sqrt{m^2-\frac{\gamma^2}{4}m^{4\alpha}}}\to 1,\quad\mbox{as }m\to\infty,
\end{multline}
since $(m+1)^{4\alpha}-m^{4\alpha}\sim c(\alpha) m^{4\alpha-1}$ as $m$ goes to infinity for some constant $c(\alpha)$. Consequently, we choose $K$ so large that we have either $\tilde{v}_k(x)=\mathbf{P}_{m(k)}u_0$ or $\tilde{v}_k(x)=\mathbf{P}_{m(k)}u_0+\mathbf{P}_{m(k)+1}u_0$.
Thus in the worst case we have 
\begin{multline}
\|\tilde{v}_k(x)\|_{L^q(\Omega)}\leq (m(k)+1)^\delta\(\|\mathbf{P}_{m(k)}u_0\|+\|\mathbf{P}_{m(k)+1}u_0\|\)\leq\\ 2^\frac\delta2(m(k)^2+1)^\frac\delta2(\|\mathbf{P}_{m(k)}u_0\|+\|\mathbf{P}_{m(k)+1}u_0\|),
\end{multline}
and hence
\begin{equation}
\|\tilde{v}_k\|^2_{L^q(\Omega)}\leq 2^{\delta+1}(m(k)^2+1)^\delta\(\|\mathbf{P}_{m(k)}u_0\|^2+\|\mathbf{P}_{m(k)+1}u_0\|^2\).
\end{equation}
Using the fact that $k^2\leq m(k)^2-\frac{\gamma^2}{4}m(k)^{4\alpha}\leq m(k)^2\leq \(m(k)+1\)^2$ and the previous inequality we derive
\begin{multline}
\label{1}
\sum_{k=K}^\infty(1+k^2)^{\frac{1}{2}-\frac{1}{q}}\|\tilde{v}_k\|^2_{L^q(\Omega)}\leq\\
2^{\delta+1}\sum_{k=K}^\infty(1+k^2)^{\frac{1}{2}-\frac{1}{q}}\(m(k)^2+1\)^\delta\(\|\mathbf{P}_{m(k)}u_0\|^2+\|\mathbf{P}_{m(k)+1}u_0\|^2\)\leq\\
2^{\delta+2}\sum_{m=m(K)}^\infty\(m^2+1\)^{\frac{1}{2}-\frac{1}{q}+\delta}\|P_mu_0\|^2. 
\end{multline}

Thus from \eqref{0.9}, \eqref{1}, \eqref{2.02} one concludes
\begin{equation}
\label{2.1}
\|\tilde{v}\|^2_{L^q([0,2\pi];\Omega)}\leq C_{\alpha,\gamma}\sum_{m=0}^\infty(1+m^2)^{\frac{1}{2}-\frac{1}{q}+\delta}\|\mathbf{P}_mu_0\|^2\sim C_{\alpha,\gamma}\|u_0\|^2_{H^{\frac{1}{2}-\frac{1}{q}+\delta}(\Omega)}.
\end{equation}
To obtain the estimate for $v(t)=e^{it\sqrt{-\Dx-\frac{\gamma^2}{4}(-\Dx)^{2\alpha}}}u_0$ we just notice that
$v(t)$ solves the equation
\begin{equation}
\begin{cases}
\Dt v-iAv=i\(\sqrt{-\Dx-\frac{\gamma^2}{4}(-\Dx)^{2\alpha}}-A\)v,\\
v(0)=u_0, \quad v|_{\partial\Omega}=0.
\end{cases}
\end{equation}
Hence by Duhamel's formula we see that
\begin{equation}
\label{2.2}
v(t)=e^{iAt}u_0+i\int_0^te^{iA(t-s)}\(\sqrt{-\Dx-\frac{\gamma^2}{4}(-\Dx)^{2\alpha}}-A\)v(s)ds.
\end{equation}
Due to \eqref{2.1} the estimate for $e^{iAt}u_0$ is immediate and the integral term in \eqref{2.2} can be estimated by Minkowski inequality (for brevity we use notation $\tilde{A}:=\(\sqrt{-\Dx-\frac{\gamma^2}{4}(-\Dx)^{2\alpha}}-A\)$)
\begin{multline}
\label{2.21}
\|\int_0^te^{iA(t-s)}\tilde{A}v(x,s)ds\|_{L^q([0,2\pi];\Omega)}=\|\|\int_0^te^{iA(t-s)}\tilde{A}v(x,s)ds\|_{L^q_x(\Omega)}\|_{L^q_t(0,2\pi)}\leq\\
\|\int_0^{2\pi}\|e^{iA(t-s)}\tilde{A}v(x,s)\|_{L^q_x(\Omega)}ds\|_{L^q_t(0,2\pi)}\leq\int_0^{2\pi}\|\|e^{iA(t-s)}\tilde{A}v(x,s)\|_{L^q_x(\Omega)}\|_{L^q_t(0,2\pi)}ds=\\
\int_0^{2\pi}\|e^{iAt}\(e^{-iAs}\tilde{A}v(x,s)\)\|_{L^q([0,2\pi];\Omega)}ds.
\end{multline}
Finally using \eqref{2.1} and the fact that $e^{-iAs}\tilde{A}e^{is\sqrt{-\Dx-\frac{\gamma^2}{4}(-\Dx)^{2\alpha}}}\in\mathcal{L}\(H^{\frac{1}{2}-\frac{1}{q}+\delta}(\Omega)\)$ we obtain ( assuming that $t\leq 2\pi$)
\begin{multline}
\label{2.22}
\|\int_0^te^{iA(t-s)}\tilde{A}v(x,s)ds\|_{L^q([0,2\pi];\Omega)}\leq \int_0^{2\pi}\|e^{-iAs}\tilde{A}v(x,s)\|_{H^{\frac{1}{2}-\frac{1}{q}+\delta}(\Omega)}ds=\\
\int_0^{2\pi}\|e^{-iAs}\tilde{A}e^{is\sqrt{-\Dx-\frac{\gamma^2}{4}(-\Dx)^{2\alpha}}}u_0\|_{H^{\frac{1}{2}-\frac{1}{q}+\delta}(\Omega)}ds\leq C\|u_0\|_{H^{\frac{1}{2}-\frac{1}{q}+\delta}(\Omega)},
\end{multline}
that finishes the proof.
\end{proof}

\begin{cor}
\label{cor homstri}
Let $u$ satisfies \eqref{eq ldwv} in the sense of distributions with $h(t)=0$ and inital  data be such that $u_0\in H^1_0(\Omega)$, $u_1\in L^2(\Omega)$. Then $u$ possesses the following space time regularity
\begin{equation}
\|u\|_{L^5([0,t];L^{10}(\Omega))}\leq C\(\|u_0\|_{H^1(\Omega)}+\|u_0\|\),\ t\in(0,2\pi],
\end{equation}
for some positive constant $C$ that depends on $\alpha$ and $\gamma$ only.
\end{cor}
\begin{proof}
First, one notices that
\begin{equation}
\label{2.3}
\|e^{it\sqrt{-\Dx-\frac{\gamma^2}{4}(-\Dx)^{2\alpha}}}u_0\|_{L^5([0,2\pi];L^{10}(\Omega))}\leq C\|u_0\|_{H^1}.
\end{equation}
Indeed, $(-\Dx)^\frac{3}{20}u_0\in H^\frac{7}{10}(\Omega)=H^{\frac{1}{2}-\frac{1}{5}+\frac{2}{5}}(\Omega)$. Hence we are able to use Theorem \ref{th Sogge} and Proposition \ref{prop eitDx est} with $q=5$, $\delta=\frac{2}{5}$ and $(-\Dx)^\frac{3}{20}u_0$ instead of $u_0$ that yields
\begin{equation}
\|e^{it\sqrt{-\Dx-\frac{\gamma^2}{4}(-\Dx)^{2\alpha}}}u_0\|_{L^5([0,2\pi];W^{\frac{3}{10},5}(\Omega))}\leq C\|u_0\|_{H^1}, 
\end{equation}
and continuous embedding $W^{\frac{3}{10},5}(\Omega)\subset L^{10}(\Omega)$ gives \eqref{2.3}. Moreover, estimate \eqref{2.3} implies that
\begin{equation}
\label{2.4}
\|\cos \( t \sqrt{ -\Dx -\frac{\gamma^2}{4}(-\Dx)^{2\alpha}}\)w\|_{L^5\([0,2\pi];L^{10}(\Omega)\)}\leq C\|w\|_{H^1(\Omega)},\ \forall w\in H^1_0(\Omega),
\end{equation}
\begin{equation}
\label{2.5}
\|\sin\(t\sqrt {-\Dx-\frac{\gamma^2}{4}(-\Dx)^{2\alpha}}\)w\|_{L^5\([0,2\pi];L^{10}(\Omega)\)}\leq C\|w\|_{H^1(\Omega)},\ \forall w\in H^1_0(\Omega),
\end{equation}
since $\sin(x)$ and $\cos(x)$ are linear combinations of $e^{ix}$ and $e^{-ix}$.
 
Second, we see that $v$ which solves equation
\begin{equation}
\label{2.6}
\begin{cases}
\Dt^2 v-\Dx v-\frac{\gamma^2}{4}(-\Dx)^{2\alpha}v=0,\\
v|_{\partial\Omega}=0, \ v(0)=u_0:=v_0\in H^1_0(\Omega),\ \Dt v(0)=\frac{\gamma}{2}(-\Dx)^\alpha u_0+u_1:=v_1\in L^2(\Omega),
\end{cases}
\end{equation} 
also satisfies the estimate
\begin{equation}
\label{2.61}
\|v\|_{L^5([0,2\pi];\Omega)}\leq C(\|u_0\|_{H^1(\Omega)}+\|u_1\|),
\end{equation}
with some positive constant $C$, due to the fact that it can be written as follows
\begin{equation}
v(t)=\cos \( t \sqrt{ -\Dx -\frac{\gamma^2}{4}(-\Dx)^{2\alpha}}\)v(0)+\frac{\sin\(t\sqrt {-\Dx-\frac{\gamma^2}{4}(-\Dx)^{2\alpha}}\)}{\sqrt {-\Dx-\frac{\gamma^2}{4}(-\Dx)^{2\alpha}}}\Dt v(0).
\end{equation}
The last step is to notice that solution $u(t)$ of homogeneous problem \eqref{eq ldwv} is related to solution $v(t)$ of \eqref{2.6} by
\begin{equation}
\label{2.7}
u(t)=e^{-\frac{\gamma}{2}(-\Delta)^\alpha t}v(t),
\end{equation} 
which actually was the initial observation how we could deduce Strichartz  estimates for damped wave equation via ordinary wave equation (that was done above). Operators $e^{-\frac{\gamma}{2}(-\Delta)^\alpha t}$, $t\geq 0$, define an analytic semigroup in $L^{10}(\Omega)$ (see Chapter IX, section 11, \cite{yos}). Thus operators $e^{-\frac{\gamma}{2}(-\Dx)^\alpha t}$ are bounded from $L^{10}(\Omega)$ to $L^{10}(\Omega)$ and we have $\|u(t)\|_{L^{10}(\Omega)}\leq\|v(t)\|_{L^{10}(\Omega)}$ that together with \eqref{2.61} completes the proof.
%\begin{multline}
%\|u(t)\|_{L^5_t([0,2\pi];L^{10}_x(\Omega))}\leq C\(\|v(t)\|_{L^5_t([0,2\pi];L^{10}(\Omega))}\)\leq \\
%C (\|v(0)\|_{H^1(\Omega)}+\|\Dt v(0)\|)\leq C (\|u_0\|_{H^1(\Omega)}+\|u_1\|).
%\end{multline}
\end{proof}  
%\section{Non-homogeneous Strichartz estimate in the linear case}
%To establish Strichartz estimate for non-homogeneous equation \eqref{eq ldwv} we use, as %usual, so called Christ-Kiselev Lemma (see Lemma 3.1 in \cite{Sogge Kisv l}, \cite{CK2001} %and references therein).
%\begin{lemma}
%\label{lem CK}
%Let $X$ and $Y$ be to Banach spaces and assume that $K(t,s)$ is  continuous function taking its values in $\Cal{B}(X,Y)$, space of linear and bounded operators from $X$ into $Y$. Suppose that $-\infty\leq a<b\leq\infty$ operator and $Tf(t)=\int_a^bK(t,s)f(s)ds$ is such that $T\in \Cal{B}(L^p\([a,b];X),L^q([a,b];Y)\)$, where $1\leq p<q\leq\infty $. Then Operator $Wf(t)=\int_a^tK(t,s)f(s)ds$ also belongs to $\Cal{B}\(L^p([a,b];X),L^q([a,b];Y)\)$.
%\end{lemma}
%Before going further we also need the next simple observation
%\begin{rem}\label{rem E-1 est}
%Let $u$ be an energy solution of homogeneous problem \eqref{eq ldwv}, that is $h\equiv 0$, with initial data $\xi_0\in\Cal E_{-1}:=L^2(\Omega)\times H^{-1}(\Omega)$. Then $u$ satisfies the following estimate
%\begin{equation}
%\label{est E-1}
%\|\xi_u(t)\|_{\Cal E_{-1}}+\|u\|_{L^2([0,t];H^\alpha(\Omega))}+\|\Dt u\|_{L^2([0,t];H^{\alpha-1}(\Omega))}\leq C\|\xi_0\|_{\Cal E_{-1}}e^{-\beta t},\ t\geq 0,
%\end{equation}
%for some $C,\beta>0$ which are independent of $t$ and $\xi_0$.
%\end{rem}
%\begin{proof}
%The statement easily follows from multiplication of \eqref{eq ldwv} by $(-\Dx)^{-1}\Dt u+\delta(-\Dx)^{\alpha-1}u$ with some small enough $\delta>0$.
%\end{proof}
Now we are ready to prove inhomogeneous Strichartz estimate in linear case.
\begin{cor}
\label{cor nonhomstri}
Let $u$ be an energy solution of problem \eqref{eq ldwv} with $h(t)\in L^1([0,2\pi];L^2(\Omega))$ where $\theta\in[0,1]$ and initial data $\xi_0\in\Cal{E}$. Then $u$ possesses the following space time regularity
\begin{equation}
\|u\|_{L^5([0,T];L^{10}(\Omega))}\leq C\(\|\xi_0\|_{\Cal E}+\|h\|_{L^1([0,T];L^2(\Omega))}\),\ T\in(0,2\pi],
\end{equation}
for some positive constant $C$ which does not depend on $T\in (0,2\pi]$ and $\xi_0$.
\end{cor}
\begin{proof}
Let us fix some $T\in(0,2\pi]$. Due to the fact that solution of inhomogeneous problem \eqref{eq ldwv} can be written as follows
\begin{multline}
u(t)=e^{-\frac\gamma2(-\Dx)^\alpha t}\cos\(t\sqrt{-\Dx-\frac{\gamma^2}{4}(-\Dx)^{2\alpha}}\)u_0+e^{-\frac\gamma2(-\Dx)^\alpha t}\frac{\sin\(t\sqrt{-\Dx-\frac{\gamma^2}{4}(-\Dx)^{2\alpha}}\)
}{\sqrt{-\Dx-\frac{\gamma^2}{4}(-\Dx)^{2\alpha}}}u_1+\\
\int_0^te^{-\frac\gamma2(-\Dx)^\alpha (t-s)}\frac{\sin\((t-s)\sqrt{-\Dx-\frac{\gamma^2}{4}(-\Dx)^{2\alpha}}\)
}{\sqrt{-\Dx-\frac{\gamma^2}{4}(-\Dx)^{2\alpha}}}h(s)ds:=R_1+R_2+R_3.
\end{multline}
and Corollary \ref{cor homstri} it remains to estimate $L^5([0,T];L^{10}(\Omega))$ norm of $R_3$. Using Minkowski inequality, the fact that $e^{-\frac{\gamma}{2}(-\Dx)^\alpha(t-s)}$ is analytic in $L^{10}(\Omega)$ we derive
\begin{multline}
\|R_3\|_{L^5([0,T];L^{10}(\Omega))}\leq \\
\|\int_0^t\|e^{-\frac\gamma2(-\Dx)^\alpha (t-s)}\frac{\sin\((t-s)\sqrt{-\Dx-\frac{\gamma^2}{4}(-\Dx)^{2\alpha}}\)}{\sqrt{-\Dx-\frac{\gamma^2}{4}(-\Dx)^{2\alpha}}}h(s)\|_{L^{10}_x(\Omega)}ds\|_{L^5_t([0,T])}\leq\\
\int_0^T\|\|\frac{\sin\((t-s)\sqrt{-\Dx-\frac{\gamma^2}{4}(-\Dx)^{2\alpha}}\)}{\sqrt{-\Dx-\frac{\gamma^2}{4}(-\Dx)^{2\alpha}}}h(s)\|_{L^{10}_x(\Omega)}\|_{L^5_t([0,T])}ds.
\end{multline}
Now let us consider $\frac{\sin\((t-s)\sqrt{-\Dx-\frac{\gamma^2}{4}(-\Dx)^{2\alpha}}\)}{\sqrt{-\Dx-\frac{\gamma^2}{4}(-\Dx)^{2\alpha}}}h(s)$ as function of $t$ for a fixed $s$. Since $\sin(x)$ is odd, inequality \eqref{2.5} implies 
\begin{equation}
\|\|\frac{\sin\((t-s)\sqrt{-\Dx-\frac{\gamma^2}{4}(-\Dx)^{2\alpha}}\)}{\sqrt{-\Dx-\frac{\gamma^2}{4}(-\Dx)^{2\alpha}}}h(s)\|_{L^{10}_x(\Omega)}\|_{L^5_t([0,T])}\leq C\|h(s)\|, \quad for\ almost\ all\ s\in[0,T],
\end{equation}
and therefore we end up with the estimate
\begin{equation}
\|R_3\|_{L^5([0,T];L^{10}(\Omega))}\leq C\int_0^T\|h(s)\|ds,
\end{equation}
that completes the proof.
\end{proof}
Due to dissipation estimate \eqref{2.0} we are able to improve Corollary \ref{cor nonhomstri} as follows
\begin{cor}
\label{cor nhstri.lin [t-1,t]}
Let assumptions of Proposition \ref{prop StrLin} hold and $u$ be an energy solution of \eqref{eq ldwv}. Then $u$ satisfies the following estimate
\begin{equation}
\|u\|_{L^5([\max\{0,t-1\},t];L^{10}(\Omega))}\leq C\(e^{-\beta t}\|\xi_0\|_\Cal{E}+\int_0^te^{-\beta(t-s)}\|h(s)\|ds\),\ t\geq 0,
\end{equation}
for some $\beta,C>0$ which are independent of $t$ and $\xi_0$. 
\end{cor}
\begin{proof}
Indeed, applying Proposition \ref{prop StrLin} on segment $[\tau(t),t]$ with $\tau(t)=\max\{0,t-1\}$ we obtain
\begin{equation}
\label{2.71}
\|u\|_{L^5([\tau(t),t];L^{10}(\Omega))}\leq C\(\|\xi_u(\tau(t))\|_\Cal{E}+\int_{\tau(t)}^t\|h(s)\|ds\).
\end{equation}
Due to Proposition \eqref{prop lin basic} we get
\begin{multline}
\|\xi_u(\tau(t))\|\leq C \(e^{-\tau(t)}\|\xi_0\|_\Cal{E}+\int_0^{\tau(t)}e^{-\beta(\tau(t)-s)}\|h(s)\|ds\)\leq \\
Ce^\beta\(e^{-\beta t}\|\xi_0\|_\Cal{E}+\int_0^te^{-\beta(t-s)}\|h(s)\|ds\).
\end{multline}
Also the second term of \eqref{2.71} can be estimated as follows
\begin{multline}
\int_{\tau(t)}^t\|h(s)\|ds\leq \int_{\tau(t)}^te^{\beta(t-s)}e^{-\beta(t-s)}\|h(s)\|ds\leq e^\beta\int_{\tau(t)}^te^{-\beta(t-s)}\|h(s)\|ds\leq\\
e^\beta\int_0^te^{-\beta(t-s)}\|h(s)\|ds.
\end{multline}
Collecting together the above estimates we complete the proof.
\end{proof}
Furthermore, due to dissipation ( see Proposition \ref{prop lin basic}) and Corollary \ref{cor nonhomstri} we obtain corresponding space-time estimate on {\it arbitrary} segment $[0,T]$ which is uniform with respect to $T$
\begin{prop}
\label{prop StrLin}
Let the assumptions of Corollary \ref{cor nonhomstri} hold, $h(t)\in L^1([0,T];L^2(\Omega))$ and $u$ be a weak solution of \eqref{eq ldwv}. Then $u$ satisfies the estimate
\begin{equation}
\label{est StrLin}
\|u\|_{L^5([0,T];L^{10}(\Omega))}\leq C (\|u_0\|_{H^1(\Omega)}+\|u_1\|+\|h(t)\|_{L^1([0,T];L^2(\Omega))}),\ T>0,
\end{equation}  
where constant $C$ is independent of $T$ and initial data and depends on $\gamma$ and $\alpha$ only.
\end{prop}  
\begin{proof}
Indeed, denoting $N:=\left[\frac{T}{2\pi}\right]$, we deduce 
\begin{multline}
\label{2.8}
\|u\|_{L^5([0,T];L^{10}(\Omega))}\leq \sum_{k=0}^{N-1}\|u\|_{L^5([2\pi k, 2\pi (k+1)];L^{10}(\Omega))}+\|u\|_{L^5([2\pi N, T];L^{10}(\Omega))}\leq\\
 C\(\sum_{k=0}^{N-1}\(\|\xi_u(2\pi k)\|_\Cal{E}+\|h(t)\|_{L^1([2\pi k;2\pi(k+1)];L^2(\Omega))}\)+\|\xi_u(2\pi N)\|_\Cal{E}+\|h(t)\|_{L^1([2\pi N,T];L^2(\Omega))}\)\leq\\
C\(\sum_{k=0}^N\(\|\xi_0\|_\Cal{E}e^{-\beta 2\pi k}+\int_0^{2\pi k}e^{-\beta(2\pi k -s)}\|h(s)\|ds\)+\|h\|_{L^1([0,T];L^2(\Omega))}\). 
\end{multline}
Obviously
\begin{equation}
\label{2.9}
\sum_{k=0}^N\|\xi_0\|_\Cal{E}e^{-\beta 2\pi k}\leq \frac{1}{1-e^{-2\pi \beta}}\|\xi_0\|_\Cal{E}.
\end{equation}
Also we have
\begin{multline}
\label{2.10}
\sum_{k=0}^N\int_0^{2\pi k}e^{-\beta(2\pi k -s)}\|h(s)\|ds=\sum_{k=1}^N\sum_{m=0}^{k-1}\int_{2\pi m}^{2\pi (m+1)}e^{-\beta(2\pi k -s)}\|h(s)\|ds\leq\\
 \sum_{k=1}^N\sum_{m=0}^{k-1}e^{-2\pi\beta(k-m-1)}\int_{2\pi m}^{2\pi (m+1)}\|h(s)\|ds=\sum_{m=0}^{N-1}\sum_{k=m+1}^Ne^{-2\pi\beta(k-m-1)}\int_{2\pi m}^{2\pi (m+1)}\|h(s)\|ds\leq\\
\sum_{m=0}^{N-1}\int_{2\pi m}^{2\pi (m+1)}\|h(s)\|ds\sum_{k=m+1}^\infty e^{-2\pi\beta(k-m-1)}\leq \frac{1}{1-e^{-2\pi \beta}}\|h\|_{L^1([0,T];L^2(\Omega))}.  
\end{multline}
Combining \eqref{2.8}-\eqref{2.10} we finish the proof.
\end{proof}

\section{Shatah-Struwe global solutions for semi-linear damped wave equation}
This section is devoted to global well-posedness of the following 
non-linear problem
\begin{equation}
\label{eq main}
\begin{cases}
\Dt^2u+\gamma(-\Dx)^\alpha\Dt u-\Delta u + f(u)=g(x), \quad x\in \Omega\\
u|_{\partial{\Omega}}=0, \quad u(0)=u_0\in H^1_0(\Omega) ,\quad \Dt u(0)=u_1\in L^2(\Omega),
\end{cases}
\end{equation}
where as before $\Omega$ is a bounded smooth domain, $\gamma>0$, $g(x)\in L^2(\Omega)$ and non-linearity $f\in C^1(\mathbf{R})$ of sub-critical growth satisfying natural dissipative assumptions
\begin{align}
\label{f ass1}
&|f'(s)|\leq C(1+|s|^q), \quad q\in [0,4);\\
\label{f ass2}
&f(s)s\geq -M;
%\label{f ass3}
%&f'(s)\geq -K.
\end{align} 
We recall that multiplying formally \eqref{eq main} by $\Dt u$ one finds natural to define weak {\it energy} solution of problem \eqref{eq main} on segment $[0,T]$ as function $u$ of the following regularity
\begin{equation}
\label{1}
\xi_u(t)\in L^\infty([0,T];\Cal{E}),\ \Dt u\in L^2([0,T];H^\alpha(\Omega))
\end{equation}
which satisfies equation \eqref{eq main} in the sense of distributions, that is
\begin{multline}
-\int_0^T(\Dt u,\Dt \phi)dt+\int_0^T(\nabla u,\nabla \phi)dt+\gamma\int_0^T(\Dt u,(-\Dx)^\alpha\Dt \phi)dt+\\
\int_0^T(f(u),\phi)dt=\int_0^T(g,\phi)dt,\ \forall\phi\in C^\infty_0((0,T)\times\Omega),
\end{multline}
and $\xi_u|_{t=0}=(u_0,u_1)$. Also, weak energy solution which in addition belongs to $L^5([0,T];L^{10}(\Omega))$ to be called as Shatah-Struwe solution of problem \eqref{eq main} on segment $[0,T]$.

Existence of weak energy solutions of problem \eqref{eq main} is known for a long time, see \cite{CV}. Based on Corollary \ref{cor nonhomstri} we are now proving existence of {\it Shatah-Struwe} solutions on arbitrary segment $[0,T]$.

%With knowledge that all solutions to linear equation \eqref{eq ldwv} possesses extra space-time regularity ( see Theorem \ref{cor nonhomstri}) we are going to prove existence of solutions to non-linear equation \eqref{eq main} with the same extra regularity treating problem \eqref{eq main} as perturbation of the linear problem \eqref{eq ldwv}. 

\begin{theorem}
\label{th StrEx}
Let $\gamma>0$, $\alpha\in(0,\frac{1}{2})$, $g\in L^2(\Omega)$ and non-linearity $f$ satisfies \eqref{f ass1},\eqref{f ass2}. Then for every $\xi_0\in\Cal{E}$ there exists Shatah-Struwe solution of equation \eqref{eq main} with initial data $\xi_0$ on arbitrary segment $[0,T]$ and the following estimate holds:
\begin{align}
\label{est dissip}
\|\xi_u(t)\|_\Cal{E}+\|\Dt u\|_{L^2([\max\{0,t-1\},t];H^\alpha(\Omega))}\leq Q(\|\xi_0\|_\Cal{E})e^{-\beta t}+ Q(\|g\|),\ t\geq 0,\\
\label{est L5L10}
\|u\|_{L^5([\max\{0,t-1\},t];L^{10}(\Omega))}\leq Q(\|\xi_0\|_{\Cal E })e^{-\beta t}+Q(\|g\|),\ t\geq 0,
\end{align}
for some constant $\beta>0$ and a monotone increasing function $Q$ which are independent of $t$.
\end{theorem}
\begin{proof}
We are going to construct the solution by Galerkin method. So let $\{e_j\}_{j=1}^\infty$ be complete orthonormal system of eigenfunctions of Laplacian and let $\Cal P_N$ be orthoprojector on the first $N$ eigenfunctions $e_j$. And let $\xi_0^N:=\Cal P_N\xi_0$ be initial data to approximate solution $u_N$, that is
\begin{equation}
\label{eq uN}
\Dt^2u_N+\gamma(-\Dx)^\alpha\Dt u_N-\Dx u_N+\Cal P_N f(u_N)=\Cal P_Ng,\ \xi_{u_N}(0)=\xi^N_0,
\end{equation}
where $u_N(t)=\sum_{j=1}^Nc_j(t)e_j$ for some unknown functions $c_j(t)$.

Multiplying \eqref{eq uN} by $\Dt u_N+\eb u_N$ with small enough $\eb$ one obtains well-known energy estimate ( see \cite{bk BV,CV,KZwvEq2009} for details)
\begin{equation}
\label{3.3}
\|\xi_{u_N}(t)\|_\Cal{E}+\|\Dt u_N\|_{L^2([\max\{0,t-1\},t];H^\alpha(\Omega))}\leq Q(\|\xi_0\|_\Cal{E})e^{-\beta t}+ Q(\|g\|),\ t\geq 0,
\end{equation}
for some $\beta>0$ and monotone increasing function $Q$. In particular, since for finite dimensional space all norms are equivalent, this means that $\xi_{u_N}(t)$ is defined globally.

Our next goal is to establish uniform $L^{5}([0,T];L^{10}(\Omega))$ estimate for $u_N(t)$. To this end it is convenient to represent $u_N$ in the form $u_N=v_N+w_N$, where $v_N$ solves
\begin{equation}
\label{eq vN}
\Dt^2 v_N+\gamma(-\Dx)^\alpha\Dt v_N-\Dx v_N=0,\ \xi_{v_N}(0)=\xi^N_0,
\end{equation}
and $w_N$ is determined by 
\begin{equation}
\label{eq wN}
\Dt^2 w_N+\gamma(-\Dx)^\alpha\Dt w_N-\Dx w_N=-P_Nf(v_N+w_N)+P_Ng,\ \xi_{w_N}(0)=0.
\end{equation}
Applying Corollary \ref{cor nhstri.lin [t-1,t]} to \eqref{eq vN} we get
\begin{equation}
\label{3.1}
\|v_N\|_{L^5([0,T];L^{10}(\Omega))}\leq Ce^{-\beta t}\|\xi^N_0\|_\Cal{E}\leq Ce^{-\beta t}\|\xi_0\|_{\Cal{E}},
\end{equation}
on arbitrary segment $[0,T]$. 

Also applying Corollary \ref{cor nonhomstri} for \eqref{eq wN}, together with \eqref{f ass1},\eqref{f ass2}, we have
\begin{multline}
\label{3.2}
\|w_N\|_{L^5([0,T];L^{10}(\Omega))}\leq T\|g\|+\|f(v_N+w_N)\|_{L^1([0,T];L^2(\Omega))}\leq T\|g\|+\\C\int_0^T\|1+|v_N+w_N|^{q+1}\|_{L^2(\Omega)}\leq T(\|g\|+C)+C\int_0^T\|v_N+w_N\|_{L^{2q+2}(\Omega)}^{q+1}dt.
\end{multline}

Thus if $q\in[0,2]$, then \eqref{3.2},\eqref{3.3} implies
\begin{equation}
\|w_N\|_{L^5([0,T];L^{10}(\Omega))}\leq T \(Q_1(\|\xi_0\|_\Cal{E})+Q_1(\|g\|)\),
\end{equation}
for some monotone increasing function $Q_1$ independent of $t$ and $\xi_0$. This together with \eqref{3.1} gives
\begin{equation}
\|u_N\|_{L^5([0,T];L^{10}(\Omega))}\leq (T+1) \(Q_1(\|\xi_0\|_\Cal{E})+Q_1(\|g\|)\),\quad q\in[0,2],\quad T\geq 0,
\end{equation}
for a monotone increasing function $Q_1$.

In case $q\in(2,4)$ we use continuous embedding $L^{10}(\Omega)\subset L^{2q+2}(\Omega)$, to continue \eqref{3.2} as follows
\begin{multline}
\|w_N\|_{L^5([0,T];L^{10}(\Omega))}\leq
 T(\|g\|+C)+C\int_0^T\|v_N\|^{q+1}_{L^{10}(\Omega)}+\|w_N\|^{q+1}_{L^{10}(\Omega)}dt\leq \\
T(\|g\|+C)+CT^\frac{4-q}{5}\|v_N\|^{q+1}_{L^5([0,T];L^{10}(\Omega))}+CT^\frac{4-q}{5}\|w_N\|^{q+1}_{L^5([0,T];L^{10}(\Omega))}\leq \\
\Big(T(\|g\|+C)+CT^\frac{4-q}{5}\|\xi_0\|^{q+1}_\Cal{E}\Big)+CT^\frac{4-q}{5}\|w_N\|^{q+1}_{L^5([0,T];L^{10}(\Omega))}\leq\\
CT^\frac{4-q}{5}(\|\xi_0\|^{q+1}_{\Cal{E}}+\|g\|+1)+C\|w_N\|^{q+1}_{L^5([0,T];L^{10}(\Omega))},
\end{multline}
where we suppose $T\leq 1$ in the last inequality. The last estimate with Young's inequality implies
\begin{equation}
\|w_N\|_{L^5([0,T];L^{10}(\Omega))}\leq T^\frac{4-q}{5}\(Q_1(\|\xi_0\|_{\Cal E})+Q_1(\|g\|)\)+C\|w_N\|^{q+1}_{L^5([0,T];L^{10}(\Omega))},\ T\leq 1,
\end{equation}
for some monotone increasing  $Q_1$ independent of $T$ and $\xi_0$. Since the problem is autonomous we also deduce the estimate on shifted segments of size $T$
\begin{multline}
\label{3.4}
\|w_N\|_{L^5([max\{0,t-T\},t];L^{10}(\Omega))}\leq T^\frac{4-q}{5}\(Q_1(\|\xi_{u_N}(\max\{0,t-T\})\|_{\Cal E})+Q_1(\|g\|)\)+\\
C\|w_N\|^{q+1}_{L^5([\max\{0,t-T\},t];L^{10}(\Omega))}\leq\\
T^\frac{4-q}{5}\(e^{-\beta t}e^\beta Q_1(\|\xi_0\|_\Cal{E})+e^\beta Q_1(\|g\|)\)+
C\|w_N\|^{q+1}_{L^5([\max\{0,t-T\},t];L^{10}(\Omega))},\ \quad t>0,\ T\in(0,1],
\end{multline}
where at the last step we used \eqref{3.3} and technical Lemma \ref{lem Q(Ae-t+B)<Q(A)e-t+Q(B)}. 

Moreover, taking into account that $q<4$, choosing $T=\delta_0=\frac{\eb}{(e^{-\beta t}e^\beta Q_1(\|\xi_0\|_{\Cal E})+e^\beta Q_1(\|g\|))^{5/(4-q)}}<1$ with $\eb$ small enough we are in position to apply technical Lemma \ref{lem y<y^sigma + eps} ( proven right after the theorem) to \eqref{3.4} that yields
\begin{equation}
\|w_N\|_{L^5([\max\{0,t-\delta_0\},t];L^{10}(\Omega))}\leq 2\delta_0^\frac{4-q}{5}e^\beta\(e^{-\beta t} Q_1(\|\xi_0\|_{\Cal E})+Q_1(\|g\|)\),\ \ \delta_0<1,\ t>0,
\end{equation}
for some monotone increasing function $Q_1$.

Denoting for brevity $\tau(t)=\max\{0,t-1\}$, representing segment $[\tau(t),t]$ as union of segments of size $\delta_0$ and using the above inequality we derive
\begin{multline}
\label{3.5}
\|w_N\|_{L^5([\tau(t),t];L^{10}(\Omega))}\leq\\
 \sum_{i=0}^{\left[\frac{1}{\delta_0}\right]-1}\|w_N\|_{L^5([\tau(t+\delta_0 i),\tau(t+\delta_0(i+1))];L^{10}(\Omega))}+
\|w_N\|_{L^5([\tau(t+\delta_0\left[\frac{1}{\delta_0}\right]),t];L^{10}(\Omega))}\leq\\
 2e^\beta\sum_{i=0}^{\left[\frac{1}{\delta_0}\right]-1}\delta_0^\frac{4-q}{5}\(e^{-\beta \tau(t+\delta_0(i+1))}Q_1(\|\xi_0\|_{\Cal E})+Q_1(\|g\|)\)+2e^\beta\delta_0^\frac{4-q}{5}\(e^{-\beta t}Q_1(\|\xi_0\|_{\Cal E})+Q_1(\|g\|)\)\leq\\
 2e^{2\beta}\sum_{i=0}^{\left[\frac{1}{\delta_0}\right]}\delta_0^\frac{4-q}{5}\(e^{-\beta t}Q_1(\|\xi_0\|_{\Cal E})+Q_1(\|g\|)\)\leq 2e^{2\beta}\eb^\frac{4-q}{5}\(\frac{1}{\delta_0}+1\)\leq 4e^{2\beta}\eb^\frac{4-q}{5}\frac{1}{\delta_0}=\\
4e^{2\beta}\eb^{-\frac{1+q}{5}}(e^{-\beta t}Q_1(\|\xi_0\|_{\Cal E})+Q_1(\|g\|))^\frac{5}{4-q}\leq e^{-\beta t}Q(\|\xi_0\|_{\Cal E})+Q(\|g\|),
\end{multline}
where at the last step we used Lemma \ref{lem Q(Ae-t+B)<Q(A)e-t+Q(B)}. 

Thus combining \eqref{3.1} and \eqref{3.5} we deduce
\begin{equation}
\label{3.6}
\|u_N\|_{L^5([\tau(t),t];L^{10}(\Omega))}\leq e^{-\beta t}Q(\|\xi_0\|_{\Cal E})+Q(\|g\|),\ t\geq 0,
\end{equation}
for some monotone increasing function $Q$ which is independent of $t$ and $\xi_0$.

Finally with uniform estimates \eqref{3.3} and \eqref{3.6} in hands it is not difficult to check that corresponding limit $u$ of $u_N$ solves \eqref{eq main} and satisfies \eqref{est dissip}, \eqref{est L5L10}.
\end{proof}

\begin{lemma}
\label{lem y<y^sigma + eps}
Let $0<C_0<\infty$ and suppose that $0\leq y(s)\in C([a,b))$
satisfies $y(a)=0$ and
\begin{equation}
\label{est y<y^sigma+eps}
y(s)\leq C_0y(s)^\sigma+\eb,
\end{equation}
for some $\sigma>1$ and $0<\eb<\frac{1}{2}\(\frac{1}{2C_0}\)^\frac{1}{\sigma-1}$. Then
\begin{equation}
y(s)\leq 2\eb,\quad s\in [a,b).
\end{equation}
\end{lemma} 
\begin{proof}
Let us consider function
\begin{equation}
k_\eb(x)=C_0x^\sigma-x+\eb=x(C_0x^{\sigma-1}-1)+\eb.
\end{equation}
We have $k_\eb(0)=\eb>0$ and $k_\eb\(\(\frac{1}{2C_0}\)^\frac{1}{\sigma-1}\)<0$ by assumptions of the lemma. 
On the other hand, by assumptions of the Lemma $h_\eb(s)=k_\eb(y(s))\geq 0$, when $s\in[a,b)$. Consequently,
since $y(a)=0$ and $y(s)\in C([a,b))$, we have $y(s)\leq \(\frac{1}{2C_0}\)^\frac{1}{\sigma-1}$ that together with 
\eqref{est y<y^sigma+eps} gives
\begin{equation}
y(s)\leq C_0y(s)\(y(s)\)^{\sigma-1}+\eb\leq\frac{1}{2}y(s)+\eb,\quad s\in[a,b),
\end{equation}
that yields the desired result.
\end{proof}

\begin{cor}
\label{cor Eid}
Let assumptions of Theorem \ref{th StrEx} be satisfied and $u$ be a Shatah-Struwe solution of equation \eqref{eq main}. Then $u$ satisfies the following energy identity
\begin{multline}
\frac{1}{2}\|\Dt u(t_2)\|^2+\frac{1}{2}\|\nabla u(t_2)\|^2+F(u(t_2))-(g,u(t_2))+\int_{t_1}^{t_2}\|(-\Dx)^\frac{\alpha}{2}\Dt u(s)\|^2ds=\\
\frac{1}{2}\|\Dt u(t_1)\|^2+\frac{1}{2}\|\nabla u(t_1)\|^2+F(u(t_1))-(g,u(t_1)),\quad \forall t_2\geq t_1\geq 0,
\end{multline}
where $F(s)=\int_0^sf(r)dr$.
\end{cor}
\begin{proof}
Indeed, since we already know that $u\in L^5([t_1,t_2];L^{10}(\Omega))$, hence $f(u)\in L^1([t_1,t_2];L^2(\Omega))$ (due to \eqref{f ass1}). Taking into account that $u\in L^\infty([t_1,t_2];L^2(\Omega))$ we conclude that product $\int_{t_1}^{t_2}(f(u(s)),\Dt u(s))ds$ makes sense. Thus we can apply projector $\Cal{P}_N$ to \eqref{eq main}, multiply the equation by $\Dt u_N$, integrate the obtained equality over $x\in\Omega$ and $t\in[t_1,t_2]$ and then pass to the limit as $N\to\infty $. Finally it remains to notice that $\int_{t_1}^{t_2}(f(u(s),\Dt u(s))ds=F(u(t_2))-F(u(t_1))$ (see \cite{TemamDS}).
\end{proof}
\begin{cor}
\label{cor xi_u C(E)}
Let assumptions of Theorem \ref{th StrEx} be satisfied and $u$ be a Shatah-Struwe solution of \eqref{eq main}. Then $\xi_u(t)\in C([0,T];\Cal{E})$ for any $T>0$.
\end{cor}
\begin{proof}
From, the fact that $\Dt u\in L^\infty([0,T];L^2(\Omega))$ and Newton's formula follows that $u\in C([0,T];L^2(\Omega))$, that together with $u\in L^\infty([0,T];H^1_0(\Omega))$ implies (see \cite{TemamDS}, Lemma 3.3) 
\begin{equation}
\label{u Cw H1}
u\in C^w([0,T];H^1_0(\Omega)).
\end{equation}
Also from \eqref{eq main} follows that $\Dt^2 u\in L^\infty([0,T];H^{-1}(\Omega))$ that, due to Newton's formula, implies $\Dt u\in C([0,T];H^{-1}(\Omega))$ that together with $\Dt u\in L^\infty([0,T];L^2(\Omega))$ gives ( see \cite{TemamDS}, Lemma 3.3)
\begin{equation}
\label{dtu Cw L2}
\Dt u\in C^w([0,T];L^2(\Omega)).
\end{equation}
Up to this moment we have not used the fact that $u$ possesses additional regularity \eqref{est L5L10} and it is valid for any energy solution. This in particular explains how we understand initial data for energy solutions.

To prove strong continuity we need to use {\it energy equality}  \eqref{cor Eid}. From the fact $F(u)\in L^\infty([0,T];L^1(\Omega))$ and its distributional derivative $\Dt F(u)=f(u)\Dt u\in L^1([0,T];L^2(\Omega))$, exactly here we need $u\in L^5([0,T];L^{10}(\Omega))$, that is $\Dt F(u)\in L^1([0,T];L^1(\Omega))$ we conclude that $F(u)\in C([0,T];L^1(\Omega))$ ( see \cite{TemamDS}, Lemma 3.1). Since now function $t\to (F(u(t)),1)$ is continuous, from energy equality and the fact that $u\in C([0,T];L^2(\Omega))$ we conclude that
\begin{equation}
\label{cont norm}
\mbox{function } t\to \|\xi_u(t)\|^2_\Cal{E} \mbox{ is continuous.}
\end{equation} This is what we need, indeed
\begin{multline}
\|\xi_u(t)-\xi_u(t_0)\|^2_\Cal{E}=\|\xi_u(t)\|^2_\Cal{E}+\|\xi_u(t_0)\|^2_\Cal{E}-2(\nabla u(t),\nabla u(t_0))-\\
2(\Dt u(t),\Dt u(t_0))\to 0,\quad \mbox{as }t\to t_0,
\end{multline}
due to \eqref{u Cw H1}, \eqref{dtu Cw L2} and \eqref{cont norm}.
\end{proof}
\begin{cor}
\label{cor u in L2(H 1+a)}
Let assumptions of Theorem \ref{th StrEx} be satisfied and $u$ be a Shatah-Struwe solution of problem \eqref{eq main}. Then, in addition, we have $u\in L^2([0,T];H^{1+\alpha}(\Omega))$, $\Dt^2 u\in L^2([0,T]; H^{\alpha-1}(\Omega))$ and the following estimates hold:
\begin{equation}
\label{est u in L2(H a+1)}
\|u(t)\|_{L^2([\max\{0,t-1\},t];H^{1+\alpha}(\Omega)}\leq Q(\|\xi_0\|_\Cal{E})e^{-\beta t}+Q(\|g\|),\ t\geq 0,
\end{equation}
\begin{equation}
\label{est dtt u in L2(H a-1)}
\|\Dt^2 u(t)\|_{L^2([\max\{0,t-1\},t]H^{\alpha-1}(\Omega))}\leq Q(\|\xi_0\|_\Cal{E})e^{-\beta t}+Q(\|g\|),\ t\geq 0, 
\end{equation} 
where constant $\beta>0$ and $Q$ is some monotone increasing function independent of $t$.
\end{cor}
\begin{proof}
Indeed, since $u\in L^5([0,T];L^{10}(\Omega))$ hence $f(u)\in L^1([0,T];L^2(\Omega))$ and to obtain \eqref{est u in L2(H a+1)} it remains to apply Corollary \ref{cor lin L2(H 1+a)} ( as before $\tau(t)=\max\{0,t-1\}$)
\begin{multline}
\label{3.61}
\|u\|_{L^2([\tau(t),t];H^{1+\alpha}(\Omega))}\leq C\(e^{-\beta t}\|\xi_0\|_{\Cal E}+\int_0^te^{-\beta(t-s)}\|f(u(s))\|ds\)\leq\\
C\(e^{-\beta t}\|\xi_0\|_{\Cal E}+1+\int_0^te^{-\beta(t-s)}\|u^q(s)\|ds\)\leq C\(e^{-\beta t}\|\xi_0\|_{\Cal E}+1+\int_0^te^{-\beta(t-s)}\|u(s)\|^5_{L^{10}(\Omega)}ds\).
\end{multline}
The last term of \eqref{3.61} can be estimated as follows
\begin{multline}
\int_0^te^{-\beta(t-s)}\|u(s)\|^5_{L^{10}(\Omega)}ds=\sum_{i=0}^{[t]-1}\int_i^{i+1}e^{-\beta(t-s)}\|u(s)\|^5_{L^{10}(\Omega)}ds+\int_{[t]}^te^{-\beta(t-s)}\|u(s)\|^5_{L^{10}(\Omega)}ds\leq\\
\sum_{i=0}^{[t]-1}e^{-\beta(t-i-1)}\|u\|^5_{L^5([i,i+1];L^{10}(\Omega))}+\|u\|^5_{L^5([[t],t];L^{10}(\Omega))}\leq\\
e^{-\beta t}\sum_{i=0}^{[t]-1}e^{\beta (i+1)}\(e^{-\beta (i+1)}Q(\|\xi_0\|_{\Cal E})+Q(\|g\|)\)^5+\(e^{-\beta t}Q(\|\xi_0\|_{\Cal E})+Q(\|g\|)\)^5.
\end{multline}
Using that for positive $a,b$ satisfy $(a+b)^5\leq 2^4(a^5+b^5)$ we proceed as follows
\begin{multline}
\label{3.62}
\int_0^te^{-\beta(t-s)}\|u(s)\|^5_{L^{10}(\Omega)}ds\leq\\
16e^{-\beta t}\sum_{i=0}^{[t]-1}\(e^{-4\beta(i+1)}Q^5(\|\xi_0\|_\Cal{E})+e^{\beta(i+1)}Q^5(\|g\|)\)+16e^{-5\beta t}Q^5(\|\xi_0\|_\Cal{E})+Q^5(\|g\|)\leq\\
e^{-\beta t}Q_1(\|\xi_0\|_\Cal{E})+Q_1(\|g\|),
\end{multline}
where $Q_1$ is a monotone increasing function which does not depend on $t$. Combining \eqref{3.61} and \eqref{3.62} we get \eqref{est u in L2(H a+1)}.

Expressing $\Dt^2 u$ from equation \eqref{eq main}, and taking into account that $\alpha\in(0,\frac{1}{2})$, we find 
\begin{multline}
\|\Dt^2 u(t)\|_{L^2([\tau(t),t];H^{\alpha-1}(\Omega))}\leq \|u(t)\|_{L^2([\tau(t),t];H^{\alpha+1}(\Omega))}+\|\Dt u(t)\|_{L^2([\tau(t),t];H^\alpha(\Omega))}+\\
\|f(u)\|^{(1-\alpha)}_{L^\infty([\tau(t),t];H^{-1}(\Omega))}\|f(u)\|^{\alpha}_{L^1([\tau(t),t];L^2(\Omega))}+\|g\|,
\end{multline}
that due to estimates \eqref{est dissip}, \eqref{est L5L10}, \eqref{est u in L2(H a+1)} easily implies \eqref{est dtt u in L2(H a-1)}.
\end{proof}
Strichartz type estimate \eqref{est L5L10} allows us to prove uniqueness of Shatah-Struwe solutions
\begin{theorem}
\label{th uniq.cont}
Let assumptions of Theorem \ref{th StrEx} hold, $u_1$ and $u_2$ be Shatah-Struwe solutions of \eqref{eq main} with initial data $\xi^1_0,\xi^2_0\in\Cal{E}$ respectively. Then the following estimate holds
\begin{equation}
\|\xi_{u_1}-\xi_{u_2}\|_{L^\infty([0,T];\Cal{E})}+\|u_1-u_2\|_{L^5([0,T];L^{10}(\Omega))}\leq Q\(T,\|\xi^i_0\|_\Cal{E},\|g\|\)\|\xi^1_0-\xi^2_0\|_\Cal{E},
\end{equation}
for some monotone increasing $Q$. In particular Shatah-Struwe solution is unique and depends continuously on initial data.
\end{theorem}  
\begin{proof}
Assume $u_1$ and $u_2$ be two Shatah-Struwe solutions of problem \eqref{eq main} with initial data $\xi^1_0$, $\xi^2_0\in\Cal{E}$ respectively. Then the difference $v=u_1-u_2$ satisfies the equation 
\begin{equation}
\label{eq u1-u2}
\begin{cases}
\Dt^2v -\Dx v+\gamma (-\Dx)^\alpha \Dt v= f(u_1)-f(u_2),\\
\xi_v(0)=\xi^1_0-\xi^2_0,\quad v|_{\partial\Omega}=0.
\end{cases}
\end{equation}
Let $\delta>0$ be fixed, small enough and to be determined below. Due to growth assumption \eqref{f ass1} one can easily check that $f(u_1)-f(u_2)$ belongs to $L^1([0,T];L^2(\Omega))$. Thus using Proposition \ref{prop StrLin} to \eqref{eq u1-u2} and $\eqref{f ass1}$ we derive
\begin{multline}
\label{3.7}
\|\xi_v\|_{L^\infty([0,\delta];\Cal{E})}+\|v\|_{L^5([0,\delta];L^{10}(\Omega))}\leq C\|\xi^1_0-\xi^2_0\|_\Cal{E}+\|f(u_1)-f(u_2)\|_{L^1([0,\delta];L^2(\Omega))}\leq\\
C\|\xi^1_0-\xi^2_0\|_\Cal{E}+\|\int_0^1f'(\lambda u_2+(1-\lambda)u_1)d\lambda\ v\|_{L^1([0,\delta];L^2(\Omega))}\leq\\
C\|\xi^1_0-\xi^2_0\|_\Cal{E}+ C\int_0^\delta(1+|u_1|^{2q}+|u_2|^{2q},|v|^2)^\frac{1}{2}dt\leq\\
C\|\xi^1_0-\xi^2_0\|_\Cal{E}+C\int_0^\delta\|v\|dt+C\sum_{i=1}^2\int_0^\delta\|u_i\|^q_{L^\frac{5q}{2}(\Omega)}\|v\|_{L^{10}(\Omega)}dt=C\|\xi^1_0-\xi^2_0\|_\Cal{E}+A+B_1+B_2.
\end{multline}
By Sobolev embedding theorem one deduces
\begin{equation}
\label{3.8}
A\leq C\int_0^\delta\|v\|_{L^{10}(\Omega)}dt\leq C\|v\|_{L^5[0,\delta];L^{10}(\Omega))}\delta^\frac{4}{5}.
\end{equation}
The fact that $q\in [0,4)$ and \eqref{est L5L10} implies that $B_i$ can be estimated as
\begin{multline}
\label{3.9}
B_i\leq C\int_0^\delta\|u_i\|^q_{L^{10}(\Omega)}\|v\|_{L^{10}(\Omega)}dt\leq C\(\int_0^\delta\|u_i\|^\frac{5q}{4}_{L^{10}(\Omega)}dt\)^\frac{4}{5}\|v\|_{L^5([0,\delta];L^{10}(\Omega))}\leq\\
C\delta^\frac{4-q}{5}\|u_i\|^q_{L^5([0,\delta];L^{10}(\Omega))}\|v\|_{L^5([0,\delta];L^{10}(\Omega))}\leq C\delta^\frac{4-q}{5}\( Q(\|\xi^i_0\|_\Cal{E})+Q(\|g\|)\)^q\|v\|_{L^5([0,\delta];L^{10}(\Omega))}.
\end{multline}
Now  choosing $\delta$ such that 
\begin{equation}
\label{delta choice}
\delta^\frac{4-q}{5}=\frac{1}{2}\frac{1}{C+C\sum_{i=1}^2\(Q(\|\xi^i_0\|_\Cal{E})+Q(\|g\|)\)^q},
\end{equation}
we see that from \eqref{3.7}-\eqref{3.9} follows that 
\begin{equation}
\label{3.10}
\|\xi_v\|_{L^\infty([0,\delta];\Cal{E})}+\|v\|_{L^5([0,\delta];L^{10}(\Omega))}\leq 2C\|\xi^1_0-\xi^2_0\|_\Cal{E}.
\end{equation}
Fixing arbitrary $T>0$ and applying \eqref{3.10} on segments $I_0=[0,\delta],\ I_1=[\delta,2\delta],\ldots,$ $I_{[T/\delta]}=[\delta[T/\delta],T]$ we find  
\begin{equation}
\label{3.11}
\|v\|_{L^\infty(I_k;\Cal{E})}\leq (2C)^{k+1}\|\xi^1_0-\xi^2_0\|_\Cal{E},\quad k=0,\ldots, [T/\delta].
\end{equation}
Hence
\begin{equation}
\|v\|_{L^\infty([0,T];\Cal{E})}\leq (2C)^{[T/\delta]+1}\|\xi^1_0-\xi^2_0\|_\Cal{E}\leq Q(T,\|\xi^i_0\|_\Cal{E},\|g\|)\|\xi^1_0-\xi^2_0\|_\Cal{E},
\end{equation}
where $Q$ is a monotone increasing function. Furthermore, from \eqref{3.10}, \eqref{3.11} follows
\begin{multline}
\label{3.12}
\|v\|_{L^5([0,T];L^{10}(\Omega))}\leq \sum_{k=0}^{[T/\delta]}\|v\|_{L^5(I_k;\Cal{E})}\leq\|\xi^1_0-\xi^2_0\|_\Cal{E}\sum_{k=0}^{[T/\delta]}(2C)^{k+1}\leq\\
(2C)^{[T/\delta]+2}\|\xi^1_0-\xi^2_0\|_\Cal{E}\leq Q(T,\|\xi^i_0\|_\Cal{E},\|g\|)\|\xi^1_0-\xi^2_0\|_\Cal{E},
\end{multline}
for some monotone increasing function $Q$ that finishes the proof.
\end{proof}
\section{Smoothing property of Shatah-Struwe solutions}

In this section we show that hyperbolic-like equation \eqref{eq main}, in fact, possesses smoothing property similar ( but weaker) to usual parabolic equations. We note that this effect also occurs when $\alpha = \frac{1}{2}$ ( see \cite{AZqdw}) and $\alpha\in\(\frac{1}{2},1\)$ ( see \cite{KZwvEq2009}).

As usual first we prove an auxiliary result which basically says that solution of \eqref{eq main} is more regular when initial data are more regular.
\begin{theorem}
\label{th sm1}
Let assumptions of Theorem \ref{th StrEx} be satisfied and let $u$ be a Shatah-Struwe solution of \eqref{eq main} with initial data such that
\begin{equation}
\xi_u(0)\in H^2(\Omega)\cap H^1_0(\Omega)\times H^1_0(\Omega):=\Cal{E}_1.
\end{equation} 
Then $\xi_u(t)\in\Cal{E}_1$ and $\xi_{\Dt u}(t)\in\Cal{E}$ for all $t\geq 0$. Furthermore, the following estimate holds
\begin{multline}
\|\xi_u(t)\|_{\Cal{E}_1}+\|\xi_{\Dt u}(t)\|_\Cal{E}+\|\Dt u(s)\|_{L^5([0,t];L^{10}(\Omega))}\leq\\
e^{\(Q(\|\xi_0\|_\Cal{E})+Q(\|g\|)\)t}\|\xi_{\Dt u}(0)\|_\Cal{E}+Q(\|\xi_0\|_\Cal{E})+Q(\|g\|),
\end{multline}
for some constant $C$ and increasing function $Q$ independent of $t$.
\end{theorem}
\begin{proof}
Below we restrict ourselves to a sketch of the proof which can be done completely rigorously using, for example, Galerkin method. Let $v:=\Dt u$ for brevity. Then $v$ solves
\begin{equation}
\label{eq dtu}
\begin{cases}
\Dt^2v-\Dx v+\gamma (-\Dx)^\alpha \Dt v +f'(u)v=0,\\
v|_{\partial\Omega}=0,\\
v(0)=\Dt u (0)=u_1,\ \Dt v(0)= \Dt^2 u(0):=\Dx u_0-\gamma(-\Dx)^\alpha u_1-f(u_0)+g.
\end{cases}
\end{equation}

From equation \eqref{eq main}, \eqref{f ass1} and the fact $H^2(\Omega)\subset L^\infty(\Omega)$ we see that 
\begin{equation}
\|\Dt^2 u(t)\|^2\leq Q(\|\xi_u(t)\|_{\Cal{E}_1})+\|g\|^2,\quad \forall t\geq 0. 
\end{equation}
for some monotone increasing function $Q$ which is independent of $t$ and initial data, that is 
\begin{equation}
\|\xi_v(t)\|^2_{\Cal{E}}\leq Q(\|\xi_u(t)\|_{\Cal{E}_1})+\|g\|^2,\quad \forall t\geq 0. 
\end{equation}
for some monotone increasing function $Q$ which is independent of $t$ and initial data. And vice versa, multiplying \eqref{eq main} by $-\Dx u$ one derives that
\begin{equation}
\|\Dx u(t)\|^2\leq C(\|\xi_v(t)\|^2_\Cal{E}+|(f(u),-\Dx u)|+\|g\|^2),\quad \forall t\geq 0,
\end{equation}
for some absolute constant $C$. Also non-linear term in the above estimate can be controlled as follows
\begin{multline}
|(f(u),-\Dx u)|\leq \|f(u)\|\|\Dx u\|\leq C\(1+\|u\|^{q+1}_{L^{10}(\Omega)}\)\|\Dx u\|\leq\\
 C\(1+\|u\|^{q+1}_{H^\frac{6}{5}(\Omega)}\)\|\Dx u\|\leq C\(1+\|\nabla u\|^\frac{4(q+1)}{5}\|\Dx u\|^\frac{q+1}{5}\)\|\Dx u\|\leq \eb\|\Dx u\|^2+C_\eb\|\nabla u\|^\frac{8(q+1)}{4-q}+C_\eb,
\end{multline}
consequently we derive
\begin{equation}
\|\xi_u(t)\|^2_{\Cal{E}_1}\leq C\(\|\xi_v(t)\|^2_\Cal{E}+Q(\|\xi_0\|_\Cal{E})+Q(\|g\|)\),\quad \forall t\geq 0,
\end{equation}
for some absolute constant $C$ and monotone increasing $Q$, and so we conclude that it is enough to control $\|\xi_v(t)\|_\Cal{E}$. 

Since $\xi_v(0)\in\Cal{E}$, applying Strichartz estimate \eqref{est StrLin} to \eqref{eq dtu} we get
\begin{equation}
\|\xi_v(t)\|_{L^\infty([0,\delta];\Cal{E})}+\|v\|_{L^5([0,\delta];L^{10}(\Omega))}\leq C\(\|\xi_v(0)\|_\Cal{E}+\int_0^\delta\|f'(u)v\|dt\).
\end{equation}
And arguing similar to \eqref{3.7}-\eqref{3.9} we deduce that 
\begin{multline}
\|\xi_v(t)\|_{L^\infty([0,\delta];\Cal{E})}+\|v\|_{L^5([0,\delta];L^{10}(\Omega))}\leq C\Big(\|\xi_v(0)\|_\Cal{E}+\delta^\frac{4}{5}\|v\|_{L^5([0,\delta];L^{10}(\Omega))}+\\ 
\delta^\frac{4-q}{5}\(Q(\|\xi_0\|_{\Cal{E}}+Q(\|g\|)\)^q\|v\|_{L^5([0,\delta];L^{10}(\Omega))}\Big).
\end{multline}
Thus for such $\delta$ that $\delta^\frac{4-q}{5}=\frac{1}{2}\frac{1}{C+C(Q(\|\xi_0\|_\Cal{E})+Q(\|g\|))^q}$ we have
\begin{equation}
\|\xi_v(t)\|_{L^\infty([0,\delta];\Cal{E})}+\|v\|_{L^5([0,\delta];L^{10}(\Omega))}\leq 2C\|\xi_v(0)\|_{\Cal{E}}.
\end{equation}

Repeating steps \eqref{3.11}-\eqref{3.12} we come up with
\begin{multline}
\|\xi_v\|_{L^\infty([0,T];\Cal{E})}+\|v\|_{L^5([0,T];L^{10}(\Omega))}\leq (2C+4C^2)(2C)^{[T/\delta]}\|\xi_v(0)\|_\Cal{E}\leq\\ 
(2C+4C^2)exp\{\frac{T}{\delta}\ln(2C)\}\|\xi_v(0)\|_\Cal{E}\leq e^{T\(Q(\|\xi_0\|_\Cal{E})+Q(\|g\|)\)}\|\xi_v(0)\|_\Cal{E},
\end{multline}
%Assuming $t\in[k\delta_0,(k+1)\delta_0]$ and repeating the above arguments we end up with
%\begin{equation}
%\|\xi_v(t)\|_\Cal{E}\leq 2C\|\xi_v(k\delta_0)\|_\Cal{E}\leq (2C)^{k+1}\|\xi_v(0)\|_\Cal{E}\leq(2C)e^{\frac{\ln(2C)}{\delta_0}t}\|\xi_v(0)\|_\Cal{E},
%\end{equation}
%and
%\begin{multline}
%\|v(s)\|_{L^5([0,t];L^{10}(\Omega))}\leq \sum_{i=0}^{k-1}\|v(s)\|_{L^5([i\delta_0,(i+1)\delta_0];L^{10}(\Omega))}+
%\|v(s)\|_{L^5([k\delta_0,t];L^{10}(\Omega))}\leq\\
%2C\sum_{i=0}^k\|\xi_v(i\delta_0)\|_\Cal{E}\leq 4C^2\|\xi_v(0)\|_\Cal{E}\sum_{i=0}^ke^{i\ln(2C)}=4C^2\|\xi_v(0)\|_\Cal{E}\frac{(2C)^{k+1}-1}{2C-1}\leq\\
%|\mbox{ suppose } C\geq 1|\leq 4C^2 (2C)^{k+1}=8C^3e^{\frac{\ln(2C)}{\delta_0}t}, 
%\end{multline}
for an increasing function $Q$ independent of $T$, that completes the proof.
\end{proof}

The next theorem gives the above mentioned smoothing property.
\begin{theorem}
\label{th sm2}
Let assumptions of Theorem \ref{th StrEx} hold and $u$ be a Shatah-Struwe solution of problem \eqref{eq main}. Then $\xi_u(t)\in \Cal{E}_1$, $\xi_{\Dt u}\in \Cal{E}$ for any $t>0$. Moreover, the following estimate holds
\begin{multline}
\label{est sm on [0,d]}
\sup_{t\in (0,\delta_0]}t^\frac{1}{\alpha}\|\xi_u(t)\|_{\Cal{E}_1}+\sup_{t\in (0,\delta_0]}t^\frac{1}{\alpha}\|\xi_{\Dt u}(t)\|_\Cal{E}+\|t^\frac{1}{\alpha}\Dt^2 u(t)\|_{L^2([0,\delta_0];H^\alpha(\Omega))}+\\
\|t^\frac{1}{\alpha}\Dt u(t)\|_{L^5([0,\delta_0];L^{10}(\Omega))}+\|t^\frac{1}{\alpha}u(t)\|_{L^2([0,\delta_0];H^{1+\alpha}(\Omega))}\leq
Q\(\|\xi_0\|_\Cal{E}\)+Q(\|g\|),
\end{multline}
where $\delta_0=\delta(\|\xi_0\|_{\Cal{E}},\|g\|)>0$ is small enough and $Q$ is some monotone increasing function independent of $t$.
\end{theorem}
\begin{proof}
As in Theorem \ref{th sm1} we first obtain estimates for $\xi_v(t)$, where $v=\Dt u$. Let $k\geq 2$ be some fixed constant that will be specified below. Then, as one can see, $t^kv(t)$ solves
\begin{equation}
\label{eq tk v}
\begin{cases}
\Dt^2(t^kv)-\Dx(t^k v)+\gamma(-\Dx)^\alpha\Dt(t^k v)=H(t),\\
t^k v(t)|_{\partial\Omega}=0,\quad t^kv(t)|_{t=0}=0,\ \Dt \(t^k v(t)\)|_{t=0}=0, 
\end{cases}
\end{equation} 
where 
\begin{multline}
\label{H(t)}
H(t)=-f'(u)t^kv+2kt^{k-1}\Dt v+k(k-1)t^{k-2}v(t)+\gamma k t^{k-1} (-\Dx)^\alpha v:=\\
H_1(t)+H_2(t)+H_3(t)+H_4(t).
\end{multline}

Similarly to Theorem \ref{th sm1} we obtain estimates on some small segment $[0,\delta]$, where $\delta<1$ will be determined below. Applying estimates \eqref{2.0}, \eqref{est StrLin}, \eqref{est lin L2(H 1+a)} to \eqref{eq tk v} we find
\begin{multline}
\|\nabla(\delta^k v(\delta))\|+\|\Dt(t^kv(t))|_{t=\delta}\|+\|\Dt(t^k v(t))\|_{L^2([0,\delta];H^\alpha(\Omega))}+\\
\|t^k v(t)\|_{L^5([0,\delta];L^{10}(\Omega))}+\|t^kv(t)\|_{L^2([0,\delta];H^{1+\alpha}(\Omega))}\leq
 C\|H(t)\|_{L^1([0,\delta];L^2(\Omega))},
\end{multline} 
which implies
\begin{multline}
\label{4.4}
\delta^k\|\nabla v(\delta)\|+\delta^k\|\Dt v(\delta)\|+\|t^k \Dt v(t)\|_{L^2([0,\delta];H^\alpha(\Omega))}+\|t^k v(t)\|_{L^5([0,\delta];L^{10}(\Omega))}+\\
\|t^kv(t)\|_{L^2([0,\delta];H^{1+\alpha}(\Omega))}\leq
 k\|v(\delta)\|+k\|v(t)\|_{L^2([0,\delta];H^\alpha(\Omega))}+C\|H(t)\|_{L^1([0,\delta];L^2(\Omega))}.
\end{multline}
From dissipative estimate \eqref{est dissip} we conclude
\begin{equation}
\label{4.5}
 k\|v(\delta)\|+k\|v(t)\|_{L^2([0,\delta];H^\alpha(\Omega))}\leq k\Big(Q(\|\xi_0\|_\Cal{E})+Q(\|g\|)\Big).
\end{equation}
Let us estimate each $H_i(t)$ separately. $H_1(t)$ can be estimated as follows ( see also \eqref{3.9})
\begin{multline}
\label{est H1}
\|H_1(t)\|_{L^1([0,\delta];L^2(\Omega))}\leq C\|t^kv(t)\|_{L^1([0,\delta];L^2(\Omega))}+C\||u(t)|^q|t^kv(t)|\|_{L^1([0,\delta];L^2(\Omega))}\leq \\
C\|v\|_{L^1([0,\delta];L^2(\Omega))}+C\int_0^\delta\|u(t)\|^q_{L^{10}(\Omega)}\|t^kv(t)\|_{L^{10}(\Omega)}dt\leq Q(\|\xi_0\|_{\Cal{E}})+Q(\|g\|)\\
C\delta^\frac{4-q}{5}\(Q(\|\xi_0\|_{\Cal{E}})+Q(\|g\|)\)^q\|t^kv(t)\|_{L^5([0,\delta];L^{10}(\Omega))}.
\end{multline}

Using interpolation $[H^{\alpha-1}(\Omega),H^\alpha(\Omega)]_{1-\alpha}=L^2(\Omega)$ we derive the estimate for $H_2(t)$
\begin{multline}
\label{est H2}
\|H_2(t)\|_{L^1([0,\delta];L^2(\Omega))}\leq 2k\int_0^\delta t^{k-1}\|\Dt v(t)\|^{\alpha}_{H^{\alpha-1}(\Omega)}\|\Dt v(t)\|^{1-\alpha}_{H^\alpha(\Omega)}dt\leq\\2k\int_0^\delta\|\Dt v(t)\|_{H^{\alpha-1}(\Omega)}dt+
2k\int_0^\delta t^\frac{k-1}{1-\alpha}\|\Dt v(t)\|_{H^\alpha(\Omega)}dt\leq 2k\delta^\frac{1}{2}\|\Dt v(t)\|_{L^2([0,\delta];H^{\alpha-1}(\Omega))}+\\
2k\delta^\frac{1}{2}\|t^\frac{k-1}{1-\alpha}\Dt v(t)\|_{L^2([0,\delta];H^\alpha(\Omega))}.
\end{multline}
Thus, since $\delta$ is small, we choose $k$ in such way that $\frac{k-1}{1-\alpha}\geq k$, that is $k\geq\frac{1}{\alpha}$.

Due to energy estimate \eqref{est dissip} the control of $H_3(t)$ is trivial
\begin{equation}
\label{est H3}
\|H_3(t)\|_{L^1([0,\delta];L^2(\Omega))}\leq Q(\|\xi_0\|_{\Cal{E}})+Q(\|g\|). 
\end{equation}

Due to interpolation $[H^\alpha(\Omega),H^{3\alpha}(\Omega)]_\frac{1}{2}=H^{2\alpha}(\Omega)$, continuous embedding $H^{\alpha+1}(\Omega)\subset H^{3\alpha}(\Omega)$ and the fact $k\geq\frac{1}{\alpha}>2$ we can estimate $H_4(t)$-term as follows
\begin{multline}
\label{est H4}
\|H_4(t)\|_{L^1([0,\delta];L^2(\Omega))}\leq \gamma k \int_0^\delta t^{k-1}\|v(t)\|^\frac{1}{2}_{H^\alpha(\Omega)}\|v(t)\|^\frac{1}{2}_{H^{1+\alpha}(\Omega)}dt\leq\\
\gamma k\int_0^\delta\|v(t)\|_{H^\alpha(\Omega)}+t^{k+(k-2)}\|v(t)\|_{H^{1+\alpha}(\Omega)}dt
\leq \gamma k\delta^\frac{1}{2}\|v(t)\|_{L^2([0,\delta];H^\alpha(\Omega))}+\\
\gamma k\delta^\frac{1}{2}\|t^kv(t)\|_{L^2([0,\delta];H^{1+\alpha}(\Omega))}.
\end{multline}
%Due to interpolation $[L^2(\Omega),H^1(\Omega)]_{2\alpha}=H^{2\alpha}(\Omega)$ we can estimate $H_4(t)$-term as follows
%\begin{multline}
%\label{est H4}
%\|H_4(t)\|_{L^1([0,\delta];L^2(\Omega))}\leq \gamma k \int_0^\delta t^{k-1}\|v(t)\|^{1-2\alpha}\|\nabla v(t)\|^{2\alpha}dt\leq\\
%Q\(\|\xi_u\|_{L^\infty([0,\infty);\Cal{E})}\)\int_0^\delta \(t^\frac{k-1}{2\alpha}\|\nabla v(t)\|\)^{2\alpha} dt
%\leq Q\(\|\xi_u\|_{L^\infty([0,\infty);\Cal{E})}\)\delta\(\sup_{t\in[0,\delta]}\{\|t^\frac{k-1}{2\alpha}\|\nabla v(t)\|\}+1\),
%\end{multline}
%and consequently our second assumption on $k$ is $\frac{k-1}{2\alpha}\geq k$, that is $k\geq\frac{1}{1-2\alpha}$ and hence we choose $k=k_\alpha=\max\{\frac{1}{\alpha},\frac{1}{1-2\alpha}\}$.

Thus plugging in estimates \eqref{4.5}-\eqref{est H4} to \eqref{4.4} and choosing $\delta=\delta_0=\delta\(\|\xi_0\|_{\Cal{E}},\|g\|\)>0$ small enough  we conclude
\begin{multline}
\delta^\frac{1}{\alpha}\|\nabla v(\delta)\|+\delta^\frac{1}{\alpha}\|\Dt v(\delta)\|+\|t^\frac{1}{\alpha} \Dt v(t)\|_{L^2([0,\delta];H^\alpha(\Omega))}+\|t^\frac{1}{\alpha} v(t)\|_{L^5([0,\delta];L^{10}(\Omega))}+\\
\|t^\frac{1}{\alpha}v(t)\|_{L^2([0,\delta];H^{1+\alpha}(\Omega))}\leq Q\(\|\xi_0\|_{\Cal{E}}\)+Q(\|g\|),\quad \mbox{as }\delta\leq \delta_0.
\end{multline}
for some monotone increasing $Q$.
\end{proof}
%\begin{multline}
%\sup_{t\in[0,\delta]}t^{k_\alpha}\|\nabla v(t)\|+\sup_{t\in[0,\delta]}t^{k_\alpha}\|\Dt v(\delta)\|+\|t^{k_\alpha} \Dt v(t)\|_{L^2([0,\delta];H^\alpha(\Omega))}+\|t^{k_\alpha} v(t)\|_{L^5([0,\delta];L^{10}(\Omega))}\leq \\
%Q\(\|\xi_u\|_{L^\infty([0,\infty);\Cal{E})}\),\quad \mbox{as }\delta\leq \delta_0.
%\end{multline}
\begin{rem}
The multiplier $t^\frac{1}{\alpha}$ in Theorem \ref{th sm2} is not optimal. Indeed, considering homogeneous problem \eqref{eq ldwv} and taking into account that $e^{-\frac{\gamma}{2}(-\Dx)^\alpha t}$ is analytic, representation $u(t)=e^{-\frac{\gamma}{2}(-\Dx)^\alpha t}v(t)$ shows that the optimal multiplier would be $t^\frac{1}{2\alpha}$. Since for our purposes this is not important we do not investigate this question further.
\end{rem}
Now we are able to prove a dissipative variant of Theorem \ref{th sm1} 
\begin{theorem}
\label{th sm3}
Let assumptions of Theorem \ref{th StrEx} hold and $u$ be a Shatah-Struwe solution of problem \eqref{eq main} with initital data $\xi_0\in\Cal E_1$. Then there holds inequality 
\begin{equation}
\label{est e1.dis}
\|\xi_u(t)\|_{\Cal{E}_1}\leq Q(\|\xi_0\|_{\Cal E_1})e^{-\beta t}+Q(\|g\|), 
\end{equation}
for some monotone increasing function $Q$.
\end{theorem}
\begin{proof}
Indeed, for $t\in[0,\delta_0]$ estimate \eqref{est e1.dis} follows from Theorem \ref{th sm1}. Also writing down estimate \eqref{est sm on [0,d]} on segment $[t,t+\delta_0]$ with $t>0$ we conclude
\begin{multline}
\delta_0^\frac{1}{\alpha}\|\xi_u(t+\delta_0)\|_{\Cal{E}_1}+\delta_0^\frac{1}{\alpha}\|\xi_{\Dt u}(t+\delta_0)\|_\Cal{E}+\|s^\frac{1}{\alpha}\Dt^2 u(s)\|_{L^2([t,t+\delta_0];H^\alpha(\Omega))}+\\
\|s^\frac{1}{\alpha}\Dt u(s)\|_{L^5([t,t+\delta_0];L^{10}(\Omega))}+\|s^\frac{1}{\alpha}\Dt u(s)\|_{L^2([t,t+\delta_0];H^{1+\alpha}(\Omega))}\leq
Q\(\|\xi_u(t)\|_\Cal{E}\)+Q(\|g\|)\leq\\
Q\(e^{-\beta t}Q(\|\xi_0\|_\Cal{E})+Q(\|g\|)\)+Q(\|g\|),\ \forall t>0,
\end{multline} 
where at the last step we used \eqref{est dissip}. The above
estimate gives the desired the result due to the lemma below. 
\end{proof}
For the convenience of the reader we present the next lemma proven in \cite{vz96}. 
\begin{lemma}
\label{lem Q(Ae-t+B)<Q(A)e-t+Q(B)}
Let $Q:\R_+\to \R_+$ be a smooth function, $L_1,L_2\in\R_+$ and $\alpha>0$. Then there exists a monotone increasing function $Q_1:\R_+\to\R_+$ such that
\begin{equation}
Q(L_1+L_2e^{-\alpha t})\leq Q_1(L_1)+Q_1(L_2)e^{-\alpha t}.
\end{equation}
\end{lemma}
\begin{proof}
By Newton's formula we have
\begin{equation}
Q(L_1+L_2e^{-\alpha t})-Q(L_1)=\int_0^1Q'(L_1+sL^2e^{-\alpha t})L_2e^{-\alpha t}\leq Q(L_1,L_2)e^{-\alpha t},\ t\geq 0,
\end{equation}
where $Q(L_1,L_2)=L_2\sup_{s\in[0,1]}|Q'(L_1+sL_2)|$. Function $Q(L_1,L_2)$ admits the estimate 
\begin{equation}
Q(L_1,L_2)\leq Q_*(L_1^2+L_2^2)\leq Q_*(2L_1^2)+Q_*(2L_2^2)=Q^1_*(L_1)+Q^1_*(L^2),
\end{equation}
where
\begin{equation}
Q_*(r)=\sup\{Q(r_1,r_2):r_1^2+r_2^2\leq r\},\quad Q^1_*(r)=Q_*(2r^2).
\end{equation}
Thus the lemma follows with $Q_1(L)=Q(L)+Q^1_*(L)$.
\end{proof}
\section{Smooth attractors for Shatah-Struwe solutinos}
This section is devoted to asymptotic behaviour of Shatah-Struwe solutions of problem \eqref{eq main}. Let us summarize and rephrase the above obtained results in the language of dynamical systems. First, thanks to Theorem \ref{th StrEx} and Theorem \ref{th uniq.cont} we are able to define dynamical system $(S_t,\Cal{E})$ with phase space $\Cal{E}$ and evolutionary operator $S_t$ defined by
\begin{equation}
\label{def St}
S_t:\Cal{E}\to\Cal{E},\quad S_t\xi_0=\xi_u(t),\quad t\geq 0,
\end{equation}
where $u$ is a unique Shatah-Struwe solution of \eqref{eq main} with initial data $\xi_0\in \Cal{E}$. We will also refer to operator $S_t$ as semi-group operator $S_t$. Second, due to Corollary \ref{cor xi_u C(E)} we see that every trajectory $t\to S_t\xi_0$ is continuous in $\Cal{E}$. Furthermore, evolutionary operator $S_t:\Cal{E}\to\Cal{E}$ is also continuous for any fixed positive $t$ due to Theorem \ref{th uniq.cont}. Thus we can say that problem \eqref{eq main} generates a continuous dynamical system \eqref{def St}. Third, the defined dynamical system is dissipative, that is it possesses a bounded absorbing set:
\begin{Def}
\label{def abs.set}
A set $D$ to be called absorbing for dynamical system $(S_t,\Cal{E})$ if for any bounded set $B\subset \Cal{E}$ there exists time $T=T(B)$ such that for all $t\geq T$ we have $S_t B\subset D$.
\end{Def}
The dissipativity clearly follows from \eqref{est dissip}.

One of the objects that, in a sense, captures the behaviour of dynamical system when $t\to\infty$ is so called global attractor. Rigorously it can be defined as follows ( see \cite{bk BV,bk_ChuLas2010,MZDafer2008,TemamDS})
\begin{Def}
\label{Def.attr} 
A set $\Cal A\subset\Cal E$ is a global attractor for the semigroup $S_t$ in $\Cal E$ if:
\par
1) The set $\Cal A$ is compact in $\Cal E$.
\par
2) The set $\Cal A$ is strictly invariant: $S_t\Cal A=\Cal A$, $t\ge0$.
\par
3) The set $\Cal A$ {\it uniformly} attracts any bounded in $\Cal E$, i.e., for any bounded set $B$ in $\Cal E$ and any neighbourhood $\Cal O(\Cal A)$ of the attractor $\Cal A$ in $\Cal E$ there is time $T=T(B,\Cal E)$ such that
$$
S_tB\subset \Cal O(\Cal A),\ \ t\geq T.
$$
\end{Def}
We notice that due to {\it uniform} attraction property $3$ of Definition \ref{Def.attr}, the fact that $\Cal{A}$ is closed, it follows that any bounded invariant set $Y\subset \Cal{E}$ is a subset of $\Cal{A}$. This is why such defined attractor is called {\it global}. Also the compactness of $\Cal A$ guarantees that the attractor is essentially thinner than a ball in $\Cal E$ since in infinite dimensional space a ball is not pre-compact.

\begin{theorem}
\label{th attr.ex}
Let the assumptions of Theorem \ref{th StrEx} hold. Then the semigroup $S_t$ in $\Cal E$ defined by \eqref{def St} of problem \eqref{eq main} possesses the global attractor $\Cal A$ which is a bounded set in $\Cal E_1$. The attractor $\Cal A$ is generated by all  trajectories of $S_t$ which are defined for all $t\in\R$ and bounded in $\Cal E$:
%$$
\begin{equation}\label{6.k}
\Cal A=\Cal K\big|_{t=0},
\end{equation}
%$$
where $\Cal K\subset C_b(\R,\Cal E)$ is the set all bounded Shatah-Struwe solutions of \eqref{eq main} defined for all $t\in\R$.
\end{theorem}
\begin{proof}
Since semigroup $S_t$ is continuous existence of the attractor follows from the dissipativity and compactness of the semigroup $S_t$ by classic result ( see \cite{bk BV, TemamDS, MZDafer2008}). Dissipativity is already discussed. Compactness of dynamical system means existence of a compact absorbing set. In our case this is a direct consequence of Theorem \ref{th sm2} and Theorem \ref{th sm3}. Indeed, from Theorem \ref{th sm2} and Theorem \ref{th sm3} we conclude that a closed ball $\Cal B_R$ in $\Cal E_1$ of sufficiently large radius $R$ will be an absorbing set. Since $\Cal E_1$ is compactly embedded into $\Cal E$ we know that $\Cal B_R$ is precompact in $\Cal{E}$. Closedness of $\Cal B_R$ in $\Cal{E}$
follows from the facts that $\Cal B_R$ is convex and $\Cal E$
is reflexive. One just should remember Mazur theorem and Banach-Alaoglu theorem.

Finally representation \eqref{6.k} is classic. Since $\Cal A$ is invariant one easily sees that every element of $\Cal A$ generates a trajectory from $\Cal K$, thus $\Cal A\subset\Cal K|_{t=0}$.
On the other hand $\Cal K|_{t=0}$ is invariant and hence $\Cal K|_{t=0}\subset A$ since $\Cal A$ is a maximal bounded invariant set ( see comments after Definition \ref{Def.attr}). 
\end{proof}

The complexity of the structure of the attractor is in a sense measured by its fractal dimension
\begin{Def} Let $K$ be a compact set in a metric space $\Cal E$. By Hausdorff criterium, for every $\eb>0$, $K$ can be covered by finitely-many balls of radius $\eb$ in $\Cal E$. Let $N_\eb(K,\Cal E)$ be the minimal number of such balls which is enough to cover $\Cal E$. Then, the fractal dimension of $K$ is defined as follows:
%$$
\begin{equation}\label{5.frac}
\dim_f(K,\Cal E):=\limsup_{\eb\to0}\frac{\log N_\eb(K,\Cal E)}{\log\frac1\eb}.
\end{equation}
%$$
\end{Def}

Basically fractal dimension is the growth exponent of required number of $\eb-$balls needed to cover a compact set when $\eb$ tends to zero. For example, it is known that for smooth finite dimensional manifolds fractal dimension coincides with ordinary dimension ( see \cite{bk BV}). In other words to cover $n-$dimensional manifold we need about $\(\frac{1}{\eb}\)^n$ balls. 

We show finite-dimensionality of the attractor by constructing exponential attractor.

\begin{Def} A set $\Cal M$ is an exponential attractor for the semigroup $S_t$ in $\Cal E$ if the following conditions are satisfied:
\par
1) The set $\Cal M$ is compact in $\Cal E$.
\par
2) The set is $\Cal M$ is semi-invariant: $S_t\Cal M\subset\Cal M$.
\par
3) The set $\Cal M$ has finite fractal dimension in $\Cal E$.
\par
4) The set $\Cal M$ attracts exponentially the images of bounded sets, i.e., for every bounded set $B$ in $\Cal E$,
%$$
\begin{equation}\label{5.expattr}
dist_{\Cal E}(S_tB,\Cal M)\le Q(\|B\|_{\Cal E})e^{-\beta t},\ \ t\ge0,
\end{equation}
%$$
for some positive $\beta$ and monotone function $Q$ which are independent of $t$.
\end{Def}

If exponential attractor $\Cal{M}$ exists it, of course, contains the global attractor. This is easily seen from the fact that $\Cal M$ uniformly attracts all bounded sets in $\Cal{E}$ and $\Cal A$ is invariant. But in comparison to global attractor it has the advantage of attracting bounded sets exponentially fast! However we have to sacrifice its strict invariance.   

For construction of exponential attractor it is convenient to consider the action of the semigroup $S_t$ on a set $\Cal B$ defined as
\begin{equation}
\Cal{B}=\cup_{t\geq 0}S_t\Cal{B}_R,
\end{equation}
where $\Cal{B}_R$ is an absorbing ball in $\Cal E_1$ from Theorem \ref{th attr.ex}. Clearly, $\Cal B$ is bounded in $\Cal{E}_1$ (due to Theorem \ref{th sm3}), $\Cal B$ is compact in $\Cal{E}$, $\Cal B$ is positively invariant $S_t \Cal B\subset \Cal B$ and hence $S_t:\Cal B\to \Cal B$.

A technical result that allows to build exponential attractor is the following proposition
\begin{prop}
\label{prop d.exp.attr}
Let assumptions of Theorem \ref{th StrEx} hold and $S_t$ be a semigroup defined by \eqref{def St}. Then for any $\xi_1,\xi_2\in\Cal B$ the following estimate is valid
\begin{equation}
\label{est |S1xi1-S2xi2|_a<L|xi1-xi2|}
\|S_1\xi_1-S_1\xi_2\|_{\Cal E_\alpha} \leq L\|\xi_1-\xi_2\|_\Cal E,
\end{equation}
where $\Cal{E}_\alpha=H^{1+\alpha}(\Omega)\cap H^1_0(\Omega)\times H^\alpha(\Omega)$ and constant $L$ is independent of $\xi_1,\xi_2\in\Cal B$. 
\end{prop}  
\begin{proof}
Let $\xi_{u_i}(t)=S_t\xi_i$, where $i=1,2$, be two trajectories starting from $\xi_i$, and $u_i$ be corresponding Shatah-Struwe solutions. Then the difference $v(t)=u_1(t)-u_2(t)$ solves
\begin{equation}
\label{eq u1-u2}
\begin{cases}
\Dt^2v-\Dx v+\gamma (-\Dx)^\alpha\Dt v =f(u_2)-f(u_1),\\
v|_{\partial\Omega}=0,\ \xi_v(0)=\xi_1-\xi_2.
\end{cases}
\end{equation}
From the fact that $\Cal B$ is positively invariant, and bounded in $\Cal{E}_1$ we conclude that $\xi_{u_i}(t)$ is bounded in $\Cal E_1$. Furthermore, using embedding $H^2(\Omega)\subset C(\Omega)$ we derive
\begin{equation}
\label{est |f(u1)-f(u2)|<C|u1-u2|}
\|f(u_1)-f(u_2)\|\leq C\|u_1-u_2\|,
\end{equation}
where $C$ is independent of time and $\xi_i\in\Cal B$.

Multiplying equation \eqref{eq u1-u2} by $\Dt v$ ( notice, that since $\xi_{u_i}(t)\in\Cal E_1$ all products make sense) we find
\begin{equation}
\label{5.1}
\frac{d}{dt}\|\xi_v(t)\|^2_\Cal{E}+\|\Dt v\|^2_{H^\alpha(\Omega)}\leq C_1\|\xi_v(t)\|^2_\Cal{E},
\end{equation}
where we used \eqref{est |f(u1)-f(u2)|<C|u1-u2|}. Applying Gronwall inequality to \eqref{5.1} we find ( in 2 steps)
\begin{equation}
\label{5.2}
\|\xi_v(t)\|^2_{\Cal E}+\int_0^t\|\Dt v(s)\|^2_{H^\alpha(\Omega)}\,ds\le C_2\|\xi_v(0)\|^2_{\Cal E}e^{Kt}
\end{equation}
for some positive $C_2$ and $K$ which are independent of $\xi_i\in\Cal B$.
Multiplying equation \eqref{eq u1-u2} by $(-\Dx)^\alpha v$ (again all products make sense) we have
\begin{multline}
\frac{d}{dt}\((\Dt v,(-\Dx)^\alpha v)+\frac{\gamma}{2}\|v\|^2_{H^{2\alpha}(\Omega)}\)+\|v\|^2_{H^{1+\alpha}(\Omega)}=\\(f(u_2)-f(u_1),(-\Dx)^\alpha v)+\|\Dt v\|^2_{H^\alpha(\Omega)}.
\end{multline}
Integrating the above inequality from $0$ to $t$, taking into account \eqref{5.2} we deduce
\begin{equation}
\int_0^t\|v(s)\|^2_{H^{1+\alpha}(\Omega)}ds\leq C_2\|\xi_v(0)\|^2_{\Cal E}e^{Kt}
\end{equation}
and, therefore,
%$$
\begin{equation}\label{5.intsm}
\int_0^t\|\xi_v(s)\|^2_{\Cal E^\alpha}\,ds\le C_2\|\xi_v(0)\|^2_{\Cal E}e^{Kt}.
\end{equation}
Finally multiplying \eqref{eq u1-u2} by $t(-\Dx)^\alpha v$, using \eqref{est |f(u1)-f(u2)|<C|u1-u2|} we get
\begin{equation}
\frac{d}{dt}\(t\|\xi_v(t)\|^2_{\Cal{E}^\alpha}\)\leq \|\xi_v(t)\|^2_{\Cal{E}^\alpha}+C_3t\|\xi_v(t)\|^2_\Cal{E},
\end{equation}
that due to \eqref{5.intsm}, \eqref{5.2} yields
\begin{equation}
t\|\xi_v(t)\|^2_{\Cal{E}^\alpha}\leq (1+t^2)C_4e^{Kt}\|\xi_v(0)\|^2_\Cal{E}.
\end{equation}
Substituting $t=1$ finishes the proof.
\end{proof}

The next theorem which establishes existence of exponential attractor can be considered as the main result of this section.
\begin{theorem}
Let assumptions of Theorem \ref{th StrEx} hold then semi-group \eqref{def St} generated by Shatah-Struwe solutions of equation \eqref{eq main} possesses exponential attractor $\Cal M$ in $\Cal{E}$ which is a bounded subset of $\Cal E_1$. 
\end{theorem}
\begin{proof}
It is known that estimate \eqref{est |S1xi1-S2xi2|_a<L|xi1-xi2|} guarantees existence of exponential attractor $\Cal{M}_d$ for discrete dynamical system $(S_1,\Cal{B})$, since embedding $\Cal{E}_\alpha\subset\Cal{E}$ is compact (see \cite{EMZ00}). Also the map $(t,\xi_0)\to S_t\xi_0$ is uniformly Lipschitz with respect to $t\in[0,1]$ and $\xi_0$ on $\Cal{B}$ in norm of $\Cal{E}$. Indeed, Lipschitz property with respect to $\xi_0$ follows from Theorem \ref{th uniq.cont}. And Lipschitz property with respect to $t$ in norm of $\Cal E$ on $\Cal B$ follows from Newton's formula and boundedness of $\Cal B$ in $\Cal E_1$. Thus exponential attractor for continuous system $(S_t,\Cal E)$ can be found by $\Cal{M}=\cup_{t\in[0,1]}S_t\Cal M_d$. 
\end{proof}

\subsection*{Acknowledgements}
The author is grateful to Dr Sergey Zelik for numerous fruitful discussions and sharing his ideas.


\begin{thebibliography}{999}
%-----------------------------------------------------------------------------------------------------------------
\bibitem{bk BV} {\it A. V. Babin, M. I. Vishik}, "Attractors of evolutionary equations", North Holland, Amsterdam, 1992.
%-------------------------------------------------------------------------------------------------------------------
\bibitem{Sogge2009}{\it M. D. Blair, H. F. Smith, C. D. Sogge}, Strichartz estimates for the wave equation on manifolds
with boundary, Ann. I. H. Poincare - AN 26, (2009), 1817-1829.
%-------------------------------------------------------------------------------------------------------------
\bibitem{stri}{\it N. Burq, G. Lebeau, and F. Planchon}, Global Existence for Energy Critical Waves
in 3D Domains. J. of AMS, {\bf vol} 21, no. 3, (2008),  831--845.
%--------------------------------------------------------------------
\bibitem{carc3}
{\it A. Carvalho, J. Cholewa, T. Dlotko}, Strongly damped
wave problems: bootstrapping and regularity of solutions. J. Diff.
Eqns. {\bf 244} (2008), no. 9, 2310--2333.
%------------------------------------------------------------------------
\bibitem{CV}
{\it V. Chepyzhov, M. Vishik}, Attractors for equations of mathematical physics.
American Mathematical Society Colloquium Publications, 49. American Mathematical Society, Providence, RI, 2002.
%----------------------------------------------------------------------------------
\bibitem{chen} {\it W. Chen and S. Holm}, Fractional Laplacian time-space models for linear and nonlinear lossy media exhibiting arbitrary frequency power-law dependency. J. Accoust. Soc. of Am., {\bf 115}, no. 4, (2004), 1424--1430.
%---------------------------------------------------------------------------------------------------
\bibitem{tri1}
{\it S. Chen and R. Triggiani},  Gevrey class semigroups arising from elastic systems with gentle dissipation: the case $0<\alpha<\frac12$. Proc. Amer. Math. Soc. {\bf 110} (1990), no. 2, 401--415.
%---------------------------------------------------------------------------------------------------
\bibitem{tri2}
{\it S. Chen and R. Triggiani},Proof of extensions of two conjectures on structural damping for elastic systems. Pacific J. Math. 136, no. 1, (1989), 15--55.
%------------------------------------------------------------------------
\bibitem{Chu2010}{\it I. Chueshov}, Global attractors for a class of Kirchhoff wave models with a structural nonlinear damping, Journal of Abstract Differential Equations and Applications,
vol. {\bf 1}, no. 1 (2010), 86--106.
%------------------------------------------------------------------------------------
\bibitem{bk_ChuLas2010}
{\it I. Chueshov, I. Lasiecka,} "Von Karman Evolution Equations", Springer, 2010.
%-----------------------------------------------------------------------------------
\bibitem{CK2001}{\it M. Christ, A. Kiselev}, Maximal functions associated to filtrations, Journal of Functional Analysis {\bf 179} (2001), 409--425.
%-------------------------------------------------------------------------
\bibitem{GraPat}{\it M. Grasselli, V. Pata}, On the damped wave equation with critical exponent, Dynamical systems and differential equations (Wilmington, NC, 2002). Discrete Contin. Dynam. Systems ( suppl.) (2003), 351--358.
%---------------------------------------------------------------------
\bibitem{EMZ00}{\it M. Efendiev, A. Miranville and S. Zelik}, Exponential attractors for a nonlinear
reaction-diffusion system in $R^3$, C. R. Math. Acad. Sci. Paris,
{\bf 330} (2000), 713--718 .
%-------------------------------------------------------------------------
\bibitem {KZwvEq2009}{\it V. Kalantarov and S. Zelik}, Finite-dimensional attractors for the quasi-linear strongly-damped wave equation, J. Differential Equation. {\bf 247} (2009), 1120--1155.
%----------------------------------------------------------------------
\bibitem{KSZ}
{\it V. Kalantarov, A. Savostianov, S. Zelik}, Attractors for damped quintic wave equations in bounded domains, Annales Henri Poincare, submitted.
%---------------------------------------------------------------------
\bibitem{kap}
{\it L. Kapitanski}, Minimal compact global attractor for a damped semilinear wave equation. Comm. Partial Differential Equations {\bf 20} (1995), no. 7-8, 1303–-1323.
  %---------------------------------------------------------------------------------------------
 \bibitem{kap1}
{\it L. Kapitanski}, Global and unique weak solutions of nonlinear wave equations. Math. Res. Lett. {\bf 1} (1994), no. 2, 211--223. 
%-----------------------------------------------------------------------
 \bibitem{plan3} {\it N. Masmoudi, F. Planchon}, On uniqueness for the critical wave equation, Comm. Partial Differential Equations 31 (2006), no. 7-9, 1099--1107.
%------------------------------------------------------------------
\bibitem{MZDafer2008}
{\it A. Miranville, S. Zelik}, "Attractors for dissipative partial differential equations in bounded and unbounded domains". In: C. M. Dafermos, M. Pokorny, eds., Handbook of Differential Equations: Evolutionary Equations, vol. 4, Amsterdam: North-Holland, 2008.
%------------------------------------------------------------------------
\bibitem{PataZel2006}
{\it V. Pata, S. Zelik}, Smooth attractors for strongly damped wave equations. Nonlinearity, {\bf 19} (2006), no. 7, 1495--1506. 
%----------------------------------------------------------------------
\bibitem{PataZelwd}
{\it V. Pata, S. Zelik}, A remark on the damped wave equation, Communications on Pure and Applied Analysis, vol. {\bf 5}, no. 3 (2006),  611--616.
%----------------------------------------------------------------------
\bibitem{AZqdw}
{\it A. Savostianov, S. Zelik}, Smooth attractors for the quintic damped wave equations with fractional damping, Asymptotic Analysis (accepted).
%----------------------------------------------------------------------
\bibitem{AZr}
{\it A. Savostianov, S. Zelik}, Recent progress in attractors for quintic wave
equations, Equadiff13 (accepted).
%----------------------------------------------------------------------
\bibitem{SS}
{\it J. Shatah, M. Struwe}, Regularity results for nonlinear wave equations, Ann. of Math., vol. {\bf 138} (1993), no. 3, 503--518. 
%--------------------------------------------------------------------
\bibitem{Lp clustrs}{\it H. F. Smith, C. D. Sogge}, On the $L^p$ norm of spectral clusters on manifolds with boundary, 
Acta Math., 198, no. 1, (2007), 107--153.
%--------------------------------------------------------------------
\bibitem{Sogge Kisv l}{\it H. F. Smith, C. D. Sogge}, Global Strichartz estimates for non-trapping perturbations of the Laplacian, Comm. Partial Differential Equations {\bf 25},
(2000), 2171 -- 2183.
%--------------------------------------------------------------------------------------------------------------------
\bibitem{bk Sogge}{\it C. D. Sogge}, "Lectures on non-linear wave equations", Second edition. International Press, Boston,
MA, 2008.
%-------------------------------------------------------------------------------------------------------------
\bibitem{bk TaoDispEq}{\it T. Tao}, Non-linear Dispersive Equations: Local and Global Analysis. CBMS Regional
Conference Series in Mathematics. Providence, RI: AMS, 2006.
%------------------------------------------------------------------
\bibitem{TemamDS}
{\it R. Temam}, "Infinite-dimensional dynamical systems in Mechanics and Physics", Springer, 1988.
%------------------------------------------------------------------------
\bibitem{tree}{\it B. Treebya and B. Cox}, Modeling power law absorption and dispersion for acoustic
propagation using the fractional Laplacian. J. Accoust. Soc. of Am., {\bf 127}, no. 5, (2010), 2741--2748.
%------------------------------------------------------------------------------
\bibitem{vz96}
{\it M. Vishik, S. Zelik}, The trajectory attractor of a non-linear elliptic system in an unbounded domain, Mat. Sbornik, {\bf 187}, no. 12, (1996), 21--56. (in Russian)   
%--------------------------------------------------------------------------------
\bibitem{yos} {\it K. Yosida}, Functional Analysis, 6th Edition, Springer, 1980.
%------------------------------------------------------------------------------
\bibitem{ZelDCDS}
{\it S. Zelik}, Asymptotic regularity of solutions of singularly perturbed damped wave equations with supercritical nonlinearities. Discrete Contin. Dyn. Syst. {\bf 11} (2004), no. 2-3, 351--392. 
\end{thebibliography}
\end{document}